\documentclass[12pt, a4paper]{amsart}
\usepackage{graphicx} 
\usepackage{amsmath,amsthm,amssymb,latexsym,a4wide,tikz,multicol,tikz-cd, calc, bm}
\usepackage{tikz-qtree,tikz-qtree-compat}
\usepackage{mathtools,stmaryrd}
\usepackage[mathscr]{euscript}

\usepackage[colorlinks]{hyperref}
\usepackage[margin=10pt,font=small,labelfont=bf, labelsep=period]{caption}
\usepackage{mathrsfs}
\usepackage{geometry} 
\usepackage[active]{srcltx}
\usepackage{cite}

\let\OLDthebibliography\thebibliography
\renewcommand\thebibliography[1]{
  \OLDthebibliography{#1}
  \setlength{\parskip}{0pt}
  \setlength{\itemsep}{0pt plus 0.2ex}
}

\usepackage{comment}
\usepackage{tikz}
\tikzset{font=\small}
\usepackage{enumerate}
\usetikzlibrary{matrix,arrows,automata,positioning,decorations.pathmorphing}
\usetikzlibrary{cd}
\usetikzlibrary{patterns}
\tikzset{main node/.style={circle,fill=white,draw,minimum size=0.1cm,inner sep=0pt},}

\newtheorem{theorem}{Theorem} [section]
\newtheorem{lemma}[theorem]{Lemma}
\newtheorem{corollary}[theorem]{Corollary}
\newtheorem{proposition}[theorem]{Proposition}
\theoremstyle{definition}
\newtheorem{definition}[theorem]{Definition}
\newtheorem{remark}[theorem]{Remark}
\newtheorem{example}[theorem]{Example}

\newcommand{\mx}[1]{#1^{\mathfrak{m}}}

\newcommand{\Cay}{\operatorname{Cay}}

\title[The free $F$-restriction semigroups]{The free $F$-restriction semigroups}

\usepackage[active]{srcltx}

\usepackage{enumerate}
\numberwithin{equation}{section}

\makeatletter
\def\l@subsection{\@tocline{2}{0pt}{2.5pc}{2.5pc}{}}
\makeatother

\begin{document}

\author{Ganna Kudryavtseva}
\address{G. Kudryavtseva: University of Ljubljana,
Faculty of Mathematics and Physics, Jadranska ulica 19, SI-1000 Ljubljana, Slovenia / Institute of Mathematics, Physics and Mechanics, Jadranska ulica 19, SI-1000 Ljubljana, Slovenia}
\email{ganna.kudryavtseva\symbol{64}fmf.uni-lj.si}
\author{Ajda Lemut Furlani}
\address{A. Lemut Furlani: Institute of Mathematics, Physics and Mechanics, Jadranska ulica 19, SI-1000 Ljubljana, Slovenia/ Faculty of Mathematics and Physics, Jadranska ulica 19, SI-1000 Ljubljana, Slovenia}
\email{ajda.lemut\symbol{64}imfm.si}

\thanks{Ganna Kudryavtseva was partially supported by the ARIS grants  P1-0288 and J1-60025. Ajda Lemut Furlani was supported by the ARIS grant P1-0288.}

\sloppy

\begin{abstract}
We provide a geometric model for the free $X$-generated $F$-restriction semigroup in the extended signature $(\cdot\,, ^+,\mx{},\lambda)$, where the unary operation $\mx{}$  maps an element $a$ to the maximum element $\mx{a}$ of its $\sigma$-class, and the constant $\lambda$ is the unique left identity. This model is based on a certain quotient of the Cayley graph expansion of the free monoid $X^*$ with respect to the extended set of generators $X\cup \overline{X^*}$, where the generators from $\overline{X^*}$ are in a bijection with the free monoid $X^*$ and serve to capture the maximum elements of $\sigma$-classes of the quotient. We also provide models for the free  $X$-generated strong and perfect $F$-restriction semigroups 
in the same extended signature. The constructed models enable us to  solve the word problems for all the free objects under consideration.

\vspace{0.1cm}

\vspace{0.1cm}

\end{abstract}

\maketitle

\section{Introduction}
The present paper continues the program of research dedicated to the study of $F$-inverse monoids and their non-regular generalizations applying geometric methods. The interest to this topic is primarily motivated by the recent solution of the finite $F$-inverse cover problem \cite{ABO25} and the modification of the construction of the Margolis-Meakin group expansion \cite{MM89} to study $F$-inverse group expansions and their quotients \cite{AKSz21,DK25,KLF24,Sz24} using Cayley graphs of groups.

Restriction and birestriction semigroups\footnote{Restriction semigroups are also referred to as {\em left restriction semigroups} or {\em weakly left $E$-ample semigroups}, while birestriction semigroups are also known as {\em two-sided restriction semigroups} or {\em weakly $E$-ample semigroups}.} are the most well-studied generalizations of inverse semigroups, see, e.g. \cite{CDEGZ23,F91,FGG09,GG00,G05,G96,GG99,GHSz17, GH09,Kud15,Jones16,Kud19,Kud25a,KLF25}.\footnote{For an overview of the early developments of the restriction semigroup theory, we refer the reader to the survey \cite{H09}} and the introduction of \cite{FGG09}. Recent works revealing their connection with  \'etale and bi\'etale topological categories \cite{CG21,Kud25,Kud25a,KL17,MC24} make restriction and birestriction semigroups a part of the active and broad research stream devoted to the study of algebras of \'etale topological groupoids and their generalizations. Because of the importance of partial functions in theoretical computer science, restriction and birestriction categories are an active avenue of research in category theory \cite{CL02,CL24,CG21,HL21}.\footnote{For an overview of the developments in category theory, which led to introducing restriction categories, we refer the reader to \cite{CL02}.}
It is therefore a natural question to investigate the structure of generalizations of $F$-inverse monoids in the class of restriction and birestriction semigroups. 

The study of special classes of $F$-restriction semigroups traces back to the works of Fountain and Gomes \cite{FG90} and of Fountain, Gomes and Gould \cite{FGG09}. $F$-birestriction monoids originated in the works of the first named author \cite{Kud15} and Jones \cite{Jones16}, while the subclass of strong $F$-birestriction monoids was, according to our search results, first considered by Fountain, Gomes and Gould in \cite{FGG09}.\footnote{What we refer to as strong $F$-birestriction monoids in \cite{KLF25} is called {\em $FA$ monoid} in \cite{FGG09}.}
Generalizing the pivotal observation by Kinyon \cite{K18} for $F$-inverse monoids (see \cite{AKSz21}), we observed in our previous paper \cite{KLF25} that $F$-birestriction monoids as well as their strong and perfect analogues form varieties of algebras in the enriched signature $(\cdot\,,^+,\mx{},1)$ where $\mx{a}$ is the maximum element of the $\sigma$-class of $a$. Appropriately adapting the methods of \cite{Kud19, KLF24}, we described the structure of the free objects of these varieties and solved the word problem for them in \cite{KLF25}. The purpose of this paper is to treat similar questions for the free $F$-restriction semigroups.
It is worth highlighting that, unlike $F$-inverse or $F$-birestriction semigroups, $F$-restriction semigroups are not necessarily monoids, but they possess a unique left identity element, $\lambda$, which coincides with the maximum projection (see Proposition \ref{prop:n25} and Example \ref{ex:ex1}). Upon adding $\lambda$ to the signature, we treat $F$-restriction semigroups as algebraic structures in the signature $(\cdot\,, ^+, \mx{}, \lambda)$ of type $(2,1,1,0)$. 

We provide a model for the free $F$-restriction semigroup ${\mathsf{FFR}}(X)$ in the extended signature $(\cdot\,, ^+, \mx{}, \lambda)$, which is based on certain subgraphs of the Cayley graph $\Cay(X^*, X\cup \overline{X^*})$ of the free monoid $X^*$ with respect to the extended set of generators $X\cup \overline{X^*}$, where $\overline{X^*}$ is a disjoint bijective copy of $X^*$.
Our model parallels the model for the free $F$-inverse monoid,  based on the Cayley graph of the free group ${\mathsf{FG}}(X)$ with respect to extended generators $X\cup \overline{{\mathsf{FG}}(X)}$, see \cite{KLF24}. The role of  the Margolis-Meakin group expansion  is taken over by its monoid analogue, which we refer to as the {\em Gould expansion} of a monoid, after Gould who first studied this construction in \cite{G96}. For models of the free objects in some related varieties, we refer the reader to \cite{AKSz21, F91, FGG09, K11, K11a, KLF24, KLF25}.

We show that ${\mathsf{FFR}}(X)$ is canonically isomorphic to the quotient of the Gould expansion $M(X^*,X\cup \overline{X^*})$ of the free monoid $X^*$ with respect to  the generators $X\cup \overline{X^*}$ by the coarsest congruence $\rho$  which makes the quotient an $X$-generated $F$-restriction semigroup with elements of $\overline{X^*}$ being the maximum elements of $\sigma$-classes. We then show that the $\rho$-class of $(A,s) \in M(X^*,X\cup \overline{X^*})$ contains a unique representative $(\Gamma, s)$, where 
$\Gamma$ is a certain $\rho$-closed finite accessible subgraph of $\Cay(X^*, X\cup \overline{X^*})$ which is obtained from $A$ by adding to it finitely many edges using a geometric construction which encodes application of the defining relations of the congruence $\rho$. The model for ${\mathsf{FFR}}(X)$ then consists of all the pairs $(\Gamma, s) \in M(X^*,X\cup \overline{X^*})$ where $\Gamma$ is $\rho$-closed. In the strong case, the model is obtained similarly by an appropriate modification of the congruence $\rho$ and the geometric construction. In the perfect case, we work with the Gould expansion $M(X^*,X\cup \overline{X})$ and, adapting the ideas of \cite{AKSz21}, base our model on finite but not necessarily connected subgraphs of $\Cay(X^*,X\cup \overline{X})$ without isolated vertices.

The paper is organized as follows. 
Section~\ref{sec:prelim} provides the necessary preliminaries.
In Section~\ref{sec:varieties}, we introduce $F$-restriction semigroups and show 
that they form a variety of algebras in the signature $(\cdot\,, ^+, \mx{}, \lambda)$.  We also identify the subvarieties of strong $F$-restriction semigroups and perfect $F$-restriction monoids.
In Section~\ref{sec:F-restr} we provide a model for the free $F$-restriction semigroup ${\mathsf{FFR}}(X)$  (see Propositions~\ref{prop:mod1}~and~\ref{prop:model}) and solve the word problem for it (see Theorem~\ref{thm:word}). 
In Section~\ref{sec:strong} we adapt the approach of the previous section and obtain a model for the free strong $F$-restriction semigroup ${\mathsf{FFR}}_s(X)$ (see  Propositions~\ref{prop:mod2} and \ref{prop:models}) and solve the word problem for it (see Theorem~\ref{thm:words}). Further, in Section~\ref{sec:perfect} we obtain the model and solve the word problem for the free perfect $F$-restriction monoid ${\mathsf{FFR}}_p(X)$, see Theorems~\ref{thm:modelp} and \ref{thm:wordp}. All the word problems are solved in at most cubic time on the complexity of the input terms.
Finally, Section~\ref{sec:ample} shows that ${\mathsf{FFR}}(X)$ and its strong and perfect analogues are left ample (see Proposition~\ref{prop:d4a}). 

\section{Preliminaries}\label{sec:prelim}
\subsection{\texorpdfstring{$X$}{X}-generated algebras}
For undefined notions in universal algebra we refer the reader to \cite{BS81, MMT87}.
By an  {\em algebra} we understand an ordered pair $(A; S)$,
where $A$ is a nonempty set and $S$ is a collection of finitary operations on $A$ (see \cite[Definition 1.1]{MMT87}).
For an algebra $(A; f_1,\ldots,f_n)$, we call $(f_1,\ldots,f_n)$ the {\em signature} of this algebra and $(\rho(f_1),\ldots,\rho(f_n))$, where $\rho(f_i)$ is the arity of $f_i$ for each $i$,  the {\em type} of this algebra.  
Whenever there is no confusion, we abbreviate $(A; S)$ simply by $A$. Morphisms, subalgebras and congruences of algebras are considered with respect to their signature. For clarity, we will sometimes refer to morphisms (resp. congruences) between algebras of signature $(f_1,\ldots,f_n)$ as $(f_1,\ldots,f_n)$-morphisms (resp. $(f_1,\ldots,f_n)$-congruences).
If $\rho$ is a congruence on an algebra $S$, we set $\rho^\natural: S \to S/\rho$ to be the natural quotient morphism.

Let $X$ be a non-empty set and $(A;S)$ an algebra. 
We say that $A$ is {\em $X$-generated} if there exists a map, called the {\em assignment map}, 
$\iota_A:X \to A$,
such that $\iota_A(X)$ generates $A$ in the signature $S$. An {\em $X$-canonical} (or simply {\em canonical} if $X$ is understood) morphism $\varphi:A\to B$ between two $X$-generated algebras of the same type additionally satisfies
$\varphi\circ\iota_A = \iota_B.$
Such a morphism is necessarily surjective and for every two $X$-generated algebras $A$ and $B$ there is at most one canonical morphism between them. The class of all $X$-generated algebras of a given type, together with canonical morphisms, forms a category which admits an initial object, {\em the absolutely free $X$-generated algebra}, $\mathbb{T}(X)$. Elements of $\mathbb{T}(X)$ are usually called {\em terms}. For
each term $t\in \mathbb{T}(X)$, its {\em value} $[t]_A$ (or simply $[t]$ when $A$ is understood) in an algebra $A$ is its image under the canonical morphism $\mathbb{T}(X)\to A$.

We recall the standard notation
$X^+$ and $X^*$ for the free $X$-generated semigroup and the free $X$-generated monoid, respectively. We denote the identity element of $X^*$ by $\lambda$.

\subsection{Restriction semigroups}
The next definition is taken from Gould and Hollings \cite{GH09} and first appeared in  \cite{JS01}.\footnote{In \cite{JS01}, restriction semigroups are called {\em (left) twisted $C$-semigroups}, and in \cite{GH09} they are also called {\em left restriction semigroups}.}
\begin{definition} (Restriction semigroup)
{\em A restriction semigroup} is an algebra $(R; \cdot\,, ^+)$ of type $(2,1)$ such that $(R;\cdot)$ is a semigroup and the following identities hold:
\begin{equation}\label{eq:id_rest}
x^+x=x,  \quad x^+y^+=y^+x^+,  \quad  (x^+y)^+=x^+y^+,  \quad  xy^+=(xy)^+x.
\end{equation}
\end{definition}

Let $R$ be a restriction semigroup. It is readily verified that $(s^+)^2 = s^+$ and  $(s^+)^+=s^+$ for all $s\in R$. The set of {\em projections} $P(R)= \{ s^+ \colon s \in R\}$ 
is closed with respect to the multiplication and is a semilattice where $e\leq f$ if and only if $e = ef=fe$.
The definition implies that the following identity holds in $R$:
\begin{equation}\label{eq:aaa1}
(xy)^+=(xy^+)^+.
\end{equation}
Combining this identity with the last identity of \eqref{eq:id_rest}, we obtain the identity:
\begin{equation}\label{eq:aaa2}
xy^+=(xy^+)^+x.
\end{equation}
The {\em natural partial order} on $R$ is given by $s \leq t$ if and only if there exists $e \in P(R)$ such that $s=et$, or, equivalently, $s = s^+t$. Note that the restriction of this partial order to $P(R)$ coincides with the previously defined partial order. For $s,t,u\in R$ and $s \leq t$ it is easy to see that $su \leq tu$, $us \leq ut$ and $s^+ \leq t^+$.

We say that a restriction semigroup is  {\em reduced}, if it has only one projection. A reduced restriction semigroup is necessarily a monoid, whose identity element, $\lambda$, is its only projection.
The congruence $\sigma$ is defined as the minimum congruence on $R$,
which identifies all the projections. That is, for $s,t\in R$ we have that $s \mathrel{\sigma}t$ if and only if there exists some $e \in P(R)$ such that $es=et$. 
The  quotient $R/\sigma$ is the maximum reduced quotient of $R$. The $\sigma$-class of an element $s \in R$ will be denoted by $[s]_\sigma$.

We will make use of the following observation.

\begin{lemma}\label{lem:n1}
Suppose that $s,t\in R$ have a common upper bound. Then $s\mathrel{\sigma} t$.
\end{lemma}

\begin{proof}
Let $u$ be a common upper bound of $s$ and $t$.
Then $s = s^+u$ and $t=t^+u$. Setting $e=s^+t^+ = t^+s^+$,  we have
$es = t^+s^+s =t^+s^+s^+u = t^+s^+u = eu$ and similarly $et =eu$; thus $es=et$ so that $s\mathrel{\sigma} t$.
\end{proof}

\begin{corollary}\label{cor:n2} If $s,t\in R$ are such that $s\leq t$ then $s\mathrel{\sigma} t$.
\end{corollary}

\begin{definition} (Proper restriction semigroup)
A restriction semigroup is said to be {\em proper} if it satisfies the following condition: if $s^+=t^+$ and $s \mathrel{\sigma} t$ then $s=t$. 
\end{definition}

Proper restriction semigroups generalize $E$-unitary inverse semigroups.

\subsection{Cayley graphs of monoids}\label{subsec:graph}
The {\em Cayley graph} $\Cay(S,X)$ of an $X$-generated monoid $S$ is a labeled directed graph with the vertex set $S$ and the edge set
$$
\{(s, x, s[x]_S)\colon s\in S, \, x\in X\},
$$
where for an edge $e=(s, x, s[x]_S)$ its {\em initial} vertex is  $\alpha(e)=s$, its {\em terminal} vertex is $\omega(e)=s[x]_S$ and its label is $l(e)=x$. An edge whose initial and terminal vertices coincide will be sometimes called a {\em loop}.

A non-empty {\em directed path} in a directed graph $\Gamma$ is a sequence of edges $p=(e_1, \dots, e_n)$, where $n\geq 1$,  such that $\alpha(e_{i+1}) = \omega(e_i)$ for all $i\in \{1,\dots, n-1\}$. 
We set $\alpha(p) = \alpha(e_1)$, $\omega(p) = \omega(e_n)$ and say that $p$ {\em starts} at $\alpha(p)$ and {\em ends} at $\omega(p)$. 
The {\em label} of $p$ is defined as
$l(p) = l(e_1) \cdots l(e_n)$ and the {\em length} of $p$ is the number $n$ of its edges. If $u$ is a vertex, there is the {\em empty path} at $u$ of length $0$, denoted by $\varepsilon_u$. We set $\alpha(\varepsilon_u)=\omega(\varepsilon_u)=u$. By $V(\Gamma)$ and $E(\Gamma)$ we denote the sets of vertices and edges of $\Gamma$.
If $p=( e_1,\dots, e_n)$ and $q=( f_1,\dots, f_k)$ are paths in $\Gamma$ such that $\omega(p) = \alpha(q)$, their {\em concatenation} $pq$ is the path $pq=( e_1,\dots, e_n, f_1,\dots, f_k)$. 

A \emph{subgraph} of $\Cay(S,X)$ is a labeled directed graph  $\Gamma$ satisfying $$V(\Gamma)\subseteq S, \quad E(\Gamma) \subseteq E(\Cay(S,X))$$
and $\alpha(e), \omega(e) \in V(\Gamma)$ for  all $e \in E(\Gamma)$. A vertex $v \in V(\Gamma)$ will be called a {\em sink} if $\alpha(e) \neq v$ for all $e\in E(\Gamma)$.
The subgraph $\langle p\rangle$
{\em spanned} by a non-empty directed path $p = ( e_1, \dots, e_n)$ is defined by
$$
V(\langle p\rangle) = \{\alpha(e_1), \ldots, \alpha(e_n), \omega(e_n)\} \quad {\text{and}} \quad
E(\langle  p  \rangle) = \{e_1, \ldots, e_n \}.
$$
The subgraph spanned by the empty path $\varepsilon_s$ has one vertex, $s$, and no edges.  
For subgraphs $\Gamma_1$ and $\Gamma_2$ of $\Cay(S,X)$, 
their {\em union} $\Gamma_1 \cup \Gamma_2$ is determined by
$V(\Gamma_1 \cup \Gamma_2) = V(\Gamma_1) \cup V(\Gamma_2)$ and 
$E(\Gamma_1 \cup \Gamma_2) = E(\Gamma_1) \cup E(\Gamma_2)$.
If $\Gamma$ is a subgraph of $\Cay(S,X)$ and $s \in S$,  the subgraph $s\Gamma$ of $\Cay(S,X)$ is defined by
$
V(s\Gamma) =\{ st \colon  t\in V(\Gamma)\}$ and $E(s\Gamma) =\{ (su,x,sv)\colon (u,x,v)\in E(\Gamma)\}$.

\subsection{The Gould expansion of a monoid}\label{sec:graph_expansion}
Let $S$ be an $X$-generated monoid, whose identity element is denoted by $\lambda$.\footnote{ We refrained from using the standard notation, $1$, for the identity element of a monoid, since $\lambda$ will also denote the only left identity in $F$-restriction semigroups, and denoting the latter by $1$ would be confusing.} 
We say that a subgraph of $\Cay(S,X)$ is {\em accessible}\footnote{In \cite{GG00,G96,Cornock_phd}, accessible subgraphs are called {\em $1$-rooted subgraphs.} Our terminology is inspired by automata theory.} if it contains the origin and, for each vertex, there exists a directed path from the origin to that vertex. By $\Gamma_{\lambda}$ we denote the subgraph of $\Cay(S,X)$ with only one vertex $\lambda$ and no edges. Let ${\mathcal X}^{\lambda}_{X}$ be the set of all finite accessible subgraphs of $\Cay(S,X)$.  It is a monoid semilattice with the identity element $\Gamma_{\lambda}$ and with $A\leq B$ if and only if $A\supseteq B$.
The meet operation of this semilattice is thus given by the union of subgraphs.
We also set ${\mathcal X}_X = {\mathcal X}^{\lambda}_X\setminus \Gamma_{\lambda}$ to be the subsemilattice of ${\mathcal X}_X^{\lambda}$ obtained by removing its top element $\Gamma_{\lambda}$. 
We define the set
$$
M^{\lambda}(S,X)  = \{(\Gamma,s)\colon \Gamma \in {\mathcal X}^{\lambda}_{X},\,\, s\in {\mathrm{V}}(\Gamma)\}
$$
with the operations $\cdot$ (usually denoted by juxtaposition) and $^+$ given by 
\begin{equation}\label{eq:operations}
(A,s)(B,t) = (A\cup sB,st), \quad (A,s)^+ = (A,\lambda).    
\end{equation}
Then $M^{\lambda}(S,X)$ is a restriction monoid with the identity element $(\Gamma_{\lambda},\lambda)$  (see \cite[Proposition 9.1.2]{Cornock_phd}). The set
$$
M(S,X)=M^{\lambda}(S,X) \setminus \{(\Gamma_{\lambda},\lambda)\}
$$ 
is closed with respect to the $\cdot$ and $^+$ operations and is a $(\cdot\,, ^+)$-subalgebra of $M^{\lambda}(S,X)$.

The construction of $M^{\lambda}(S,X)$ originated in \cite{G96},\footnote{In \cite{G96}, $M^{\lambda}(S,X)$ is denoted by $\mathcal{M}(X,f,S)$, where $S$ is the $X$-generated monoid via the assignment map $f$.} where it was considered for right cancellative monoids. It was then extended to unipotent monoids (that is, monoids which have only  one idempotent) in \cite{GG00} and to arbitrary monoids in \cite{Cornock_phd}. The proof of the following statement for $M^{\lambda}(S,X)$ can be found in \cite[Proposition 3.3]{G96} (for the case where $S$ is right cancellative), see also \cite{GG00,Cornock_phd}. The adaptation to $M(S,X)$ is clear.

\begin{proposition}\label{prop:properties}
Let $(A,s), (B,t)$ be elements of  $M^{\lambda}(S,X)$ or of $M(S,X)$. Then: 
\begin{enumerate}
\item $(A,s)$ is a projection if and only if $s=\lambda$;
\item $(A,s)\leq (B,t)$ if and only if $s=t$ and $A\supseteq B$;
\item $(A,s) \mathrel{\sigma} (B,t)$ if and only if $s=t$;
\item $M^{\lambda}(S,X)$ is a proper restriction monoid and $M(S,X)$ is a proper restriction semigroup.
\end{enumerate}
\end{proposition}

For each $x \in X$, let $\Gamma_x$ denote the subgraph of $\Cay(S,X)$ with $V(\Gamma_x)=\{\lambda,[x]_S\}$ and $E(\Gamma_x)=\{(\lambda,x,[x]_S)\}$, see Figure \ref{fig:aa1}.
\begin{figure}[!ht]
\centering
\begin{tikzpicture}[scale=0.5]
\begin{scope}[every node/.style={circle,fill,inner sep=1.5pt}]
\node[label= {[label distance=-0.7cm]:$\lambda$}] (A) at (0,0) {};
\node[label= {[label distance=-0.9cm]:{$[x]_S$}}] (B) at (5,0) {}; 
\end{scope}

\begin{scope}[>={Stealth[black]}, every node/.style={fill=white,circle,scale=0.7}, every edge/.style={draw, thick}]
\path [->] (A) edge[font=\large,bend left=25] node[above=0.5pt] {$x$} (B);
\end{scope}
\end{tikzpicture}
\caption{The graph $\Gamma_x$.} \label{fig:aa1}
\end{figure}
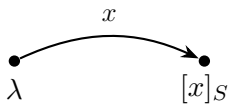

The following is proved as Proposition 9.2.6 of Cornock's  PhD thesis \cite{Cornock_phd}, and its special cases have appeared in  \cite[Proposition 3.4]{G96} and \cite[Proposition 3.4]{GG00}. We include the proof, since it provides a way of representing elements of $M(S,X)$ in terms of generators from (the image of) $X$ which will be useful in the sequel.

\begin{proposition}\label{prop:M_generators}
The restriction semigroup $M(S,X)$ and the restriction monoid $M^{\lambda}(S,X)$ are $X$-generated via the assignment  map  $\iota\colon x\mapsto (\Gamma_x,[x]_S)$. 
\end{proposition}

\begin{proof}
Let $(A,s)\in M(S,X)$. We first aim to express $(A,\lambda)$ via the generators $\iota(X)$
using the operations $\cdot$ and $^+$.  Suppose that $e=(u,x,u[x]_S)$ is an edge of $A$. Since $A$ is accessible, there is a directed path $(e_1,\dots, e_n)$ (possibly empty) from $\lambda$ to $u$. Let $x_1\cdots x_n\in X^*$ be the label of this path. Setting $p_e = (e_1,\dots, e_n,e)$,
the definition of the operations in $M(S,X)$ implies that:
\begin{equation}\label{eq:aaa5} (\langle p_e \rangle,\lambda) = 
\left((\Gamma_{x_1},[x_1]_S)\cdots ( \Gamma_{x_n}, [x_n]_S)(\Gamma_{x},[x]_S)\right)^+.
\end{equation}
Since 
$A= \bigcup_{e\in E(A)} \langle p_e\rangle$, it follows that 
\begin{equation}\label{eq:n26c}
(A,\lambda)= \prod_{e\in E(A)} (\langle p_e\rangle,\lambda).    
\end{equation}
If $s=\lambda$, we are done. Otherwise, let $y_1\cdots y_m \in X^+$ be the label of a directed path from $\lambda$ to $s$. Then we have:
\begin{equation}\label{eq:n26d}
(A,s) = (A,\lambda)( \Gamma_{y_1},[y_1]_S)\cdots ( \Gamma_{y_m},[y_m]_S),    
\end{equation}
which completes the proof of the claim for $M(S,X)$. For $M^{\lambda}(S,X)$ we  additionally consider the element $(\Gamma_{\lambda},\lambda)$. But this element is the identity element, so it is given by the nullary operation of  $M^{\lambda}(S,X)$.
\end{proof}

A special case of the following result for $M^{\lambda}(S,X)$, where $\nu$ is the identity morphism, was proved in \cite[Proposition 9.2.7]{Cornock_phd}, see also \cite[Theorem 4.2]{G96} and \cite[Proposition 3.5]{GG00}. We omit the proof as it requires minimal adjustment of the proof given in \cite[Proposition 9.2.7]{Cornock_phd} and \cite[Theorem 4.2]{G96}.

\begin{theorem} [Universal property of $M^{\lambda}(S,X)$ and $M(S,X)$.] \label{thm:universal_property_M}
Suppose $S$ is an $X$-generated monoid.
\begin{enumerate}
\item Let $R$ be an $X$-generated proper restriction monoid such that there is a canonical $(\cdot\,, ^+,\lambda)$-morphism $\nu : S \to R/\sigma$. Then there is a canonical $(\cdot\,,^+,\lambda)$-morphism  $\varphi: M \to R$ where $M=M^{\lambda}(S,X)$ such that the following diagram commutes:
$$
\begin{tikzcd}
M & R\\
S & {R/\sigma}
\arrow["\varphi", from=1-1, to=1-2]
\arrow[from=1-1, to=2-1]
\arrow[from=1-2, to=2-2]
\arrow["\nu"', from=2-1, to=2-2]\,\,.
\end{tikzcd}
$$
\item Let $R$ be an $X$-generated proper restriction semigroup such that there is a canonical $(\cdot\,, ^+)$-morphism $\nu : S \to R/\sigma$. Then there is a canonical $(\cdot\,,^+)$-morphism  $\varphi: M \to R$ where $M=M(S,X)$ such that the 
diagram above commutes. 
\end{enumerate}
\end{theorem}

We refer to the restriction monoid $M^{\lambda}(S,X)$ and the restriction semigroup $M(S,X)$ as   the {\em Gould expansions} of $S$, after Gould, who introduced  $M^{\lambda}(S,X)$ and proved its universal property in \cite{G96}.

In what follows we sometimes identify elements of $X$ with their images in $M(S,X)$ or in $M^{\lambda}(S,X)$. For example, saying that $x$ is an element $M(S,X)$ we mean its image $(\Gamma_x,[x]_S)$ under the assignment map $\iota$.

\section{Varieties of \texorpdfstring{$F$}{F}-restriction semigroups}\label{sec:varieties}
We now define the central objects of our study, $F$-restriction semigroups. Special cases of this definition have appeared in \cite{FG90, FGG99, FGG09}.
 
\begin{definition}($F$-restriction semigroup)
A restriction semigroup $R$ will be called {\em $F$-restriction} if every its $\sigma$-class contains a maximum element with respect to the natural partial order. We denote the maximum element of the $\sigma$-class of $s\in R$ by $\mx{s}$.
\end{definition}

Part (1) of the following proposition is similar to \cite[Lemma 5]{Kud15} for $F$-birestriction monoids.

\begin{proposition} \label{prop:n25}
Let $R$ be an $F$-restriction semigroup. Then:
\begin{enumerate}
\item $R$ is proper;
\item  the maximum projection in $R$, denoted by $\lambda$, is the unique left identity in $R$.
\end{enumerate}
\end{proposition}
\begin{proof}
(1) Let $R$ be an $F$-restriction semigroup and  $s,t\in R$ be such that $s\mathrel{\sigma} t$ and $s^+=t^+$. Since $s,t \leq \mx{s}=\mx{t}$, we have $s = s^+\mx{s} = t^+\mx{t} = t$. Hence $R$ is proper.

(2) The definition of $\sigma$ implies that all projections of $R$ are $\sigma$-related. Let $s\in R$ be such that $s\mathrel{\sigma} e$ for some projection $e$. Since $\sigma$ is a congruence, it follows that $s^+\mathrel{\sigma} e$, which yields $s\mathrel{\sigma} s^+$. Since $s^+ = (s^+)^+$ and $R$ is proper, we conclude that $s=s^+$. Therefore, the $\sigma$-class of a projection equals $P(R)$. Given that $R$ is $F$-restriction, it has the maximum projection, $\lambda$. Let $s\in R$, then 
$s = s^+s \leq \lambda s \leq s$, which means that $s=\lambda s$, so that $\lambda$ is a left identity for $R$.

We are left to prove that $\lambda$ is the unique left identity in $R$. Suppose that $\tilde{\lambda}$ is another left identity. Then $\lambda = \tilde{\lambda}\lambda \le \tilde{\lambda}$. In view of Corollary \ref{cor:n2}, it follows that $\lambda \mathrel{\sigma} \tilde{\lambda}$, so that $\tilde{\lambda}$ is a projection.
The maximality of $\lambda$ yields $\lambda \geq \tilde{\lambda}$. Hence $\tilde{\lambda} = \lambda$, as desired.
\end{proof}

It is known that an $F$-birestriction semigroup (and in particular any $F$-inverse semigroup) is necessarily a monoid (see, e.g., \cite[Lemma 5]{Kud15}) where the identity element is the maximum projection. The following example shows that, in the setting of $F$-restriction semigroups, the maximum projection $\lambda$ is not necessarily an identity.

\begin{example}\label{ex:ex1}
Let $S = \{x,\lambda\}$ be the two-element semilattice with $x\leq \lambda$. It is an $X$-generated monoid, where $X=\{x\}$. 
Then the semilattice $\mathcal{X}_{X}$ consists of two graphs, $\Gamma$ and $\Delta$, presented below.

\begin{figure}[ht!]
\centering
\begin{minipage}{0.45\textwidth}
\centering
\begin{tikzpicture}[scale=0.8]
\begin{scope}[every node/.style={circle,fill,inner sep=1.5pt}]
\node[label= {[label distance=-0.7cm]:$\lambda$}] (A) at (0,0) {};
\node[label= {[label distance=-0.9cm]:{$[x]_S$}}] (B) at (5,0) {}; 
\end{scope}

\begin{scope}[>={Stealth[black]}, every node/.style={fill=white,circle,scale=0.7}, every edge/.style={draw, thick}]
\path [->] (A) edge[font=\large, bend left=25] node[above=0.5pt] {$x$} (B);
\end{scope}
\end{tikzpicture}
\caption{The graph $\Gamma$.}
\end{minipage}
\hfill
\begin{minipage}{0.45\textwidth}
\centering
\begin{tikzpicture}[scale=0.8]
\begin{scope}[every node/.style={circle,fill,inner sep=1.5pt}]
\node[label= {[label distance=-0.7cm]:$\lambda$}] (A) at (0,0) {};
\node[label= {[label distance=-0.9cm]:{$[x]_S$}}] (B) at (5,0) {}; 
\end{scope}

\begin{scope}[>={Stealth[black]}, every node/.style={fill=white,circle,scale=0.7}, every edge/.style={draw, thick}]
\path [->] (A) edge[font=\large,bend left=25] node[above=0.5pt] {$x$} (B);
\path[->] (B) edge[ font=\large, out=120,in=30,looseness=20]node[above=1pt] {$x$} (B);
\end{scope}
\end{tikzpicture}
\caption{The graph $\Delta$.}
\end{minipage}
\end{figure}
It is readily seen that $M(S,X)$ has two $\sigma$-classes with two elements each and that $M(S,X)$ is an $F$-restriction semigroup with the left identity   $(\Gamma,\lambda)$. Since $(\Gamma,x)(\Gamma,\lambda) = (\Delta,x)$, we conclude that $(\Gamma,\lambda)$ is not a right identity. Therefore, $M(S,X)$ is not a monoid.
\end{example}

In the following example, we recall the construction of a semidirect product of a semilattice by a monoid from \cite[Proposition 3.1]{GG99}.\footnote{Note that the construction in \cite{GG99} is provided under the assumption that $S$ has only one idempotent, but it works for an arbitrary monoid $S$.} We will then observe in Proposition~\ref{prop:F_semidir} that this construction provides  natural families of $F$-restriction semigroups and monoids.

\begin{example} (Semidirect product of a semilattice by a monoid) \label{ex:semidir} Let $S$ be a monoid with the identity $\lambda$ and $E$ a semilattice with the meet operation $\wedge$. Suppose that $S$ acts on $E$ by endomorphisms, that is:
\begin{enumerate}
\item $(st)\cdot e = s\cdot (t\cdot e)$ for all $s,t\in S$ and $e\in E$.
\item $\lambda\cdot e = e$ for all $e\in E$; 
\item $s\cdot (e \wedge f) = (s\cdot e)\wedge (s\cdot f)$ for all $s\in S$ and $e,f\in E$.
\end{enumerate}
Let $E\rtimes S$ be the set $E\times S$ with the operations $\cdot$ and $^+$ given by:
 $$
(e,s)(f,t) = (e\wedge (s\cdot f), st), \quad (e,s)^+ = (e,\lambda).
$$
The natural partial order is given by $(e,s)\leq (f,t)$ if and only if $s=t$ and $e\leq f$, and the relation $\sigma$ is given by $(e,s) \mathrel{\sigma} (f,t)$ if and only if $s=t$. Moreover, $E\rtimes S$ is a proper restriction semigroup. 

Furthermore, $E\rtimes S$ is a monoid if and only if $E$ is a monoid and the condition
\begin{equation}\label{eq:condition}
\text{for all } s\in S: \quad s\cdot 1 = 1
\end{equation}
holds, where $1$ is the identity of $E$ (which is precisely its maximum element). In the case where  $E\rtimes S$ is a monoid, its identity element  is $(1,\lambda)$.
\end{example}

As a direct consequence, we have the following.

\begin{proposition}\label{prop:F_semidir}
 Let $S$ be a monoid with the identity $\lambda$ and $E$ a monoid semilattice with the top element $1$.  Suppose that $S$ acts on $E$ by endomorphisms. Then $E\rtimes S$  is an $F$-restriction semigroup with 
 $$
 \mx{(e,s)} = (1,s) \quad \text{ for all } (e,s)\in E\rtimes S.
 $$
The maximum projection and the unique left identity is the element $(1,\lambda)$. Furthermore, $E\rtimes S$ is an $F$-restriction monoid if and only if condition  \eqref{eq:condition} holds.
\end{proposition}

The following lemma can be proved similarly to \cite[Proposition 3.1]{AKSz21}.

\begin{lemma}\label{lem:varietyF}
An algebra $(R;\,\cdot \,,^+,\, \mx{})$ of type $(2,1,1)$ is an $F$-restriction semigroup  if and only if the algebra $(R;\, \cdot\,, ^+)$ is a restriction semigroup and the following conditions hold:
\begin{itemize}
\item[(M1)] $\mx{a} \geq a$, for all $a \in R$;
\item[(M2)] $\mx{a}=\mx{(ea)}$, for all $a\in R$ and $e\in P(R)$. 
\end{itemize}
\end{lemma}

It follows that $F$-restriction semigroups form a variety of algebras in the enriched signature $(\cdot\,, ^+, \, \mx{})$.  
By analogy with $F$-inverse and $F$-birestriction monoids we will consider  $F$-restriction semigroups in the enriched signature $(\cdot\,, ^+, \, \mx{},\lambda)$ of type $(2,1,1,0)$, where $\lambda$ is the left identity. These objects form a variety, denoted by ${\mathbf{FR}}$, which is defined by all the identities defining the variety of restriction semigroups with the left identity $\lambda$, along with the identities
\begin{equation}\label{eq:Frestr_id}
x=x^+\mx{x}, \qquad \mx{x}=\mx{(y^+x)}.
\end{equation}
$F$-restriction monoids form the subvariety ${\mathbf{FRM}}$ of the variety ${\mathbf{FR}}$
defined by adding the identity $x\lambda = x$. It follows that free $F$-restriction semigroups and free $F$-restriction monoids exist.

Note that every monoid $S$ can be viewed as an $F$-restriction monoid, that is, in the signature $(\cdot\,,^+,\mx{}, \lambda)$, where $s^+=\lambda$ and $\mx{s} = s$ for all $s\in S$.

\begin{remark} \label{rem:generated}
Let $S$ be an $X$-generated monoid with the identity element $\lambda$
and let $R$ be an $F$-restriction semigroup such that $R/\sigma$ is canonically  $(\cdot\,, ^+,\mx{}, \lambda)$-isomorphic to $S$. Suppose $R$ is 
$X$-generated in the signature $(\cdot\,, ^+,\mx{}, \lambda)$.
For each $s\in S$, let $\overline{s}$ denote the element $\mx{t}$ where $t\in \sigma^{-1}(s)$. Then $s\mapsto \overline{s}$ is a bijection from $S$ to the set $\overline{S} = \{\overline{s}\colon s\in S\} \subseteq R$. Since any $\mx{u}$, where $u\in R$, equals $\overline{\sigma^\natural(u)}$, 
we conclude that
$R$ is $(X\cup \overline{S})$-generated as a restriction semigroup.
\end{remark}

\begin{definition}(Strong and perfect $F$-restriction semigroups)
\label{def:strong_perf}
An $F$-restriction semigroup $R$ will be called {\em strong}, if 
\begin{equation}\label{eq:strong}
\mx{s}\mx{t}=(\mx{s})^+\mx{(st)} \quad \text{ for all }  s,t \in R.
\end{equation}

It will be called {\em perfect} if 
\begin{equation}\label{eq:perfect}
\mx{s}\mx{t}=\mx{(st)} \quad \text{ and }  \quad (\mx{s})^+=\lambda \quad \text{ for all } s,t \in R.
\end{equation}
\end{definition}
By the definition, strong and perfect $F$-restriction semigroups form subvarieties, which we denote by ${\mathbf{FR}}_s$ and ${\mathbf{FR}}_p$, of the variety ${\mathbf{FR}}$. These varieties are defined by the identities which define the variety ${\mathbf{FR}}$ along with the identities in \eqref{eq:strong} or \eqref{eq:perfect}, respectively. Similarly as above, we conclude that free objects in these varieties exist. The free $X$-generated objects in the varieties  ${\mathbf{FR}}$, ${\mathbf{FR}}_s$ and ${\mathbf{FR}}_p$ will be denoted by $\mathsf{FFR}(X)$, $\mathsf{FFR}_s(X)$ and $\mathsf{FFR}_p(X)$, respectively.
\begin{proposition}\label{prop:perfmonoid}
A perfect $F$-restriction semigroup is necessarily a monoid.    
\end{proposition}

\begin{proof}
Let $R$ be a perfect $F$-restriction semigroup. By Proposition \ref{prop:n25}, the maximum projection $\lambda$ is a left identity, so it suffices to show that it is a right identity.
Let $x\in R$. Since $x \mathrel{\sigma} x\lambda$, we have that $\mx{x} = \mx{(x\lambda)} = \mx{x}\mx{\lambda} = \mx{x}\lambda$. It follows that 
$x=x^+\mx{x}= x^+\mx{x}\lambda =  x\lambda$, as desired.
\end{proof}

Proposition  \ref{prop:perfmonoid} implies that the varieties of perfect $F$-restriction semigroups and of perfect $F$-restriction monoids coincide, while we will show in Section \ref{sec:strong} that free strong $F$-restriction semigroups are not monoids, so that strong $F$-restriction monoids form a proper subvariety of the variety ${\mathbf{FR}}_s$.

\begin{theorem} (Structure of perfect $F$-restriction monoids) \label{th:perf_structure}\mbox{}
\begin{enumerate}
\item  Let $S$ be a monoid with the identity $\lambda$ and $E$ a monoid semilattice with the top element $1$.  Suppose that $S$ acts on $E$ by endomorphisms. Then the $F$-restriction semigroup $E\rtimes S$ is strong if and only if it is perfect if and only if 
condition \eqref{eq:condition} holds. 
\item Suppose that $R$ is a perfect $F$-restriction monoid and let $S=R/\sigma$. For each $s\in S$ let $\overline{s}\in R$ be the maximum element of the $\sigma$-class $\sigma^{-1}(s)$. Then $S$ acts on $P(R)$ by endomorphisms via the assignment 
\begin{equation}\label{eq:endo}
s\cdot e = (\overline{s}e)^+\quad    \quad \text{ for all } s\in S \text{ and } e\in P(R).
\end{equation}
Furthermore, this action satisfies 
condition \eqref{eq:condition} and the map $R\to P(R)\rtimes S$ given by
\begin{equation}\label{eq:isom1}
r\mapsto (r^+, \sigma^{\natural}(r))
\end{equation}
is a $(\cdot\,,^+,\mx{},\lambda)$-isomorphism.
\end{enumerate}
\end{theorem}

\begin{proof}
(1)  Observe that $(\mx{(e,s)})^+ = (1,s)^+ = (1,\lambda)$, so that the second condition of \eqref{eq:perfect}
holds. But then the first condition of \eqref{eq:perfect} is equivalent to \eqref{eq:strong}, which proves that $E\rtimes S$ is strong if and only if it is perfect.

We now prove that the condition \eqref{eq:condition}  holds if and only if $E\rtimes S$ is perfect. Suppose that condition \eqref{eq:condition} holds. As already observed, the maximum projection $(1,\lambda)$ is the identity element.  Since
$$
\mx{(e,s)}\mx{(f,t)}=(1,s)(1,t)= (1\wedge (s\cdot 1),st)=(1,st) = \mx{((e,s)(f,t))},
$$ 
both of the conditions in \eqref{eq:perfect}  hold, so that
$E\rtimes S$ is perfect. Conversely, 
suppose that $E\rtimes S$ is perfect. By Proposition \ref{prop:perfmonoid}, it is a monoid. Since $(1,\lambda)$ is the left identity of $E\rtimes S$, it must be the identity. If $s\in S$, we have $(1,s)(1,\lambda) = (1\wedge (s\cdot 1), s) = (s\cdot 1,s)$. Since $(1,\lambda)$ is the identity, it follows that $s\cdot 1 = 1$.

(2) We denote the identity element of $S$ by $\lambda_S$ and the identity element of $R$ by $\lambda_R$. Then we have $\overline{\lambda_S}=\lambda_R$. Since $(\overline{st}e)^+ = (\overline{s}(\overline{t}e)^+)^+$, $\lambda_S\cdot e = (\lambda_R e)^+ = e$ and $(\overline{s}ef)^+ = (\overline{s}e)^+(\overline{s}f)^+$ for all $s,t\in S$ and $e,f\in P(R)$, \eqref{eq:endo} defines an action of $S$ on $P(R)$ by endomorphisms. In addition,  $s\cdot \lambda_R = (\overline{s}\lambda_R)^+ = \overline{s}^+$, which is equal to $\lambda_R$ by the second equality in \eqref{eq:perfect}. Hence the action satisfies condition \eqref{eq:condition}.
 A direct calculation shows that the assignment in \eqref{eq:isom1} is an $(\cdot\,,^+,\mx{},\lambda)$-isomorphism.
\end{proof}

\begin{remark} Theorem \ref{th:perf_structure} is parallel to the fact that perfect $F$-inverse monoids are precisely semidirect products of monoid semilattices by groups, where groups act on monoid semilattices by automorphisms. This fact can be proved adapting the proof of Theorem \ref{th:perf_structure} and is known (though we did not find its explicit mention in the literature).
\end{remark}

Strong $F$-restriction monoids were defined in \cite{FGG09},\footnote{Strong $F$-restriction monoids are termed {\em weakly left $FA$} in \cite{FGG09}.} and some of their special subclasses were treated in \cite{FG90, FGG99}.
Since any $F$-inverse monoid is automatically strong (cf. \cite[Remark 2.10]{KLF25}), both $F$-restriction semigroups and strong $F$-restriction semigroups generalize $F$-inverse monoids. Perfect $F$-restriction monoids generalize perfect $F$-inverse monoids \cite{AKSz21}. If $R$ is a perfect $F$-inverse monoid, then the first condition of \eqref{eq:perfect} implies the second, which is not the case for $F$-restriction semigroups or $F$-birestriction monoids.\footnote{In \cite{KLF25}, we defined perfect $F$-birestriction monoids only by the first condition of \eqref{eq:perfect}. In Definition~\ref{def:strong_perf} we imposed both of the conditions, to ensure that not only the product of two $\sigma$-classes is a whole $\sigma$-class, but also the $^+$ of a $\sigma$-class is again a whole $\sigma$-class. See also Remark \ref{rem:p}.}

\section{The free \texorpdfstring{$X$}{X}-generated \texorpdfstring{$F$}{F}-restriction semigroup}
\label{sec:F-restr}

In this section we provide a model of the free $X$-generated $F$-restriction semigroup $\mathsf{FFR}(X)$ based on the Cayley graph $\Cay(X^*, X\cup \overline{X^*})$ and prove that $\mathsf{FFR}(X)$ has solvable word problem. 

We start with the following fact which is similar to \cite[Lemma 7.1]{KLF25}.
\begin{lemma}\label{lem:free} 
$\mathsf{FFR}(X)/\sigma$ is $X$-canonically isomorphic to $X^*$.  
\end{lemma}

\begin{proof}
As we observed earlier, any $X$-generated monoid is $F$-restriction  where $u^+=\lambda$ and $\mx{u}=u$ for all $u\in S$. Let $S$ be an $X$-generated monoid. The universal property of $\mathsf{FFR}(X)$ yields that $S$ is a $(\cdot\,,^+,\mx{},\lambda)$-quotient of $\mathsf{FFR}(X)$. The needed conclusion now follows from the universal property of $X^*$.
\end{proof}

By Remark \ref{rem:generated} we have that $\mathsf{FFR}(X)$ is $(X \cup \overline{X^*})$-generated as a restriction semigroup, via the assignment map:
$$x\mapsto x, \, x\in X; \quad \overline{u} \mapsto \mx{u}, \, u\in X^*.
$$

\begin{proposition}  \label{prop:nov27a}
$\mathsf{FFR}(X)$ is an $(X\cup \overline{X^*})$-canonical $(\cdot\,,^+)$-quotient of $M(X^*, X\cup \overline{X^*})$.  
\end{proposition}

\begin{proof}
Since $\mathsf{FFR}(X)$ is a proper $(X \cup \overline{X^*})$-generated restriction semigroup with the maximum reduced quotient $X^*$, the statement follows from the universal property of $M(X^*, X \cup \overline{X^*})$, see Theorem~\ref{thm:universal_property_M}.
\end{proof}

Let $\rho$ be the $(\cdot\,, ^+)$-congruence on $M(X^*, X\cup \overline{X^*})$
such that $M(X^*, X\cup \overline{X^*})/\rho$ is $(X\cup \overline{X^*})$-canonically isomorphic to $\mathsf{FFR}(X)$. In view of Lemma \ref{lem:free} and the observation above, we have that $\rho\subseteq \sigma$. 

We now describe a generating set for $\rho$. Since $\rho$ is the coarsest $(\cdot\,, ^+)$-congruence on $M(X^*, X\cup \overline{X^*})$ which makes the elements of $\overline{X^*}$ the maximum elements of their $\sigma$-classes, $M(X^*, X\cup \overline{X^*})/\rho$ must satisfy the following conditions: 
\begin{enumerate}
\item $\overline{\lambda} \geq \overline{u}^+$ for all $\overline{u}\in \overline{X^+}$;
\item $\overline{x}\geq x$ for all $x\in X$;
\item $\overline{uv} \geq \overline{u}\, \overline{v}$ for all $\overline{u},\overline{v}\in \overline{X^+}$.
\end{enumerate}

\begin{lemma}
Let $\overline{u}\in \overline{X^*}$. Then $\overline{\lambda} \geq \overline{u}^+$ holds in $M(X^*, X\cup \overline{X^*})$ if and only if $\overline{u} = \overline{\lambda}\overline{u}$  holds in $M(X^*, X\cup \overline{X^*})$.
\end{lemma}

\begin{proof}
Suppose $\overline{\lambda} \geq \overline{u}^+$. Then $\overline{u} = \overline{u}^+\overline{u} \leq \overline{\lambda}\overline{u}$. Since $\overline{\lambda} = (\Gamma_{\overline{\lambda}},\lambda)$ is a projection, we also have $\overline{\lambda}\overline{u}\leq \overline{u}$. Thus $\overline{u} = \overline{\lambda}\overline{u}$.

Conversely, suppose $\overline{u} = \overline{\lambda}\overline{u}$. Then $\overline{u}^+ = (\overline{\lambda}\overline{u})^+ = \overline{\lambda}\overline{u}^+$, which means that $\overline{\lambda} \geq \overline{u}^+$.
\end{proof}

We consider the following relations on $M(X^*, X\cup \overline{X^*})$:
\begin{align*}
\to_{N_1} & =\{(\overline{u}, \overline{\lambda}\overline{u}) \colon u\in X^+\};\\
\to_{N_2} & =  \{(x,x^+\overline{x})\colon x\in X\};\\
\to_{N_3} & = \{(\overline{u}\,\overline{v}, (\overline{u}\,\overline{v})^+\overline{uv})\colon 
u,v\in X^+\}.
\end{align*}
and we denote the union of $\to_{N_1}$, $\to_{N_2}$ and $\to_{N_3}$ by $\to_N$.
Since all the relations $\to_{N_i}$ are clearly contained in the congruence $\rho$, the relation $\to_N$ is also contained in  $\rho$. Therefore, the congruence $\rho_N$, generated by  $\to_N$, is also contained in $\rho$. As usual, we call elements of $\to_N$ {\em rewriting rules}.
Applying induction, conditions (2) and (3) imply that for all $n\geq 2$ and for all $u_1,\dots, u_n, u\in X^*$ the quotient of $M(X^*, X\cup \overline{X^*})$ by $\rho_N$ satisfies the conditions:
\begin{equation}\label{eq:n26a}
 \overline{u_1\cdots u_n} \;\geq\; \overline{u_1}\cdots \overline{u_n}\quad\text{and}\quad \overline{u}
\;\geq\;u.   
\end{equation}

The following statement can be proved similarly as \cite[Lemma 4.1]{G96}.

\begin{lemma} \label{lem:generation}
Let $R$ be an $X$-generated restriction semigroup and $u\in R$. 
\begin{enumerate}
\item If $u$ is not a projection, it can be written as $u=ev$ where $e$ is a projection and $v$ is written via generators from $X$ using only the product operation.
\item If $u$ is a projection, it can be written as $u= u_1^+\cdots u_n^+$ where $n\geq 1$ and each $a_i$ is written via generators from $X$ using only the product operation.
\end{enumerate}
\end{lemma}

\begin{proposition}\label{prop:mod1}
$M(X^*, X\cup \overline{X^*})/\rho_N$ is an $X$-generated $F$-restriction semigroup. Consequently, $\rho_N=\rho$ and thus $M(X^*, X\cup \overline{X^*})/\rho_N$ is $X$-canonically $(\cdot\,, ^+,\mx{},\lambda)$-isomorphic to ${\mathsf{FFR}}(X)$.
\end{proposition}

\begin{proof}
Let $u\in M(X^*, X\cup \overline{X^*})/\rho_N$. We aim to show that its $\sigma$-class has a maximum element. There are two possible cases:

{\em Case 1.} Suppose first that $u$ is not a projection. 
By Lemma \ref{lem:generation}(1) we can write $u=ev$, where $v$ is written in generators from $X\cup \overline{X^*}$ using only the product operation.
Since  $x\leq \overline{x}$ for all factors $x$ in this product, where $x\in X$, 
and $\leq$ is compatible with the multiplication, we have that
$v \leq \overline{v_1}\, \overline{v_2}\, \cdots \overline{v_n}$ holds in $M(X^*, X\cup \overline{X^*})/\rho_N$ for appropriate $v_i \in X^*$. 
Applying \eqref{eq:n26a}, we have that $v \leq \overline{v_1v_2\cdots v_n}$.
Then
$$
u =ev\leq v \leq \overline{v_1v_2\cdots v_n} = \overline{\sigma^{\natural}(u)}.
$$
This shows that $\overline{\sigma^{\natural}(u)}$ is the maximum element in $[u]_\sigma$.

{\em Case 2.} Suppose now that $u$ is a projection.
By Lemma \ref{lem:generation}(2) we have that $u=u_1^+\cdots u_n^+$ where each $u_i$ is a non-empty word written in generators $X\cup \overline{X^*}$ using only the product operation. From what we proved in the previous case, it follows that $u_i\leq \overline{\sigma^{\natural}(u_i)}$ and, therefore, $ u_i^+\leq  \overline{\sigma^{\natural}(u_i)}^+\leq \overline{\lambda}$ for all $i$.  Hence $u \leq \overline{\lambda}$. 

Since $\to_N$ is contained in  $\sigma$, the congruence $\rho_N$ is also contained in $\sigma$. Hence, the quotient of $M(X^*, X\cup \overline{X^*})/\rho_N$ by $\sigma$ is isomorphic to $X^*$. Consequently, the elements $\overline{\sigma^{\natural}(u)}$, where $u\in X^*$, are maximum elements in their $\sigma$-classes in $M(X^*, X\cup \overline{X^*})/\rho_N$. This
finishes the proof of the  claim that $M(X^*, X\cup \overline{X^*})/\rho_N$ is an $F$-restriction semigroup.

Further, $M(X^*, X\cup \overline{X^*})/\rho_N$ is $X$-generated as an $F$-restriction semigroup, because it is $(X\cup \overline{X^*})$-generated as a restriction semigroup and we have that $\overline{w} = \mx{w}$ for all 
$w\in X^+$ and $\overline{\lambda} = \mx{(x^+)}$ for all $x\in X$.

Since $M(X^*, X\cup \overline{X^*})/\rho_N$ is $F$-restriction, it follows that $\rho\subseteq \rho_N$. Given that the opposite inclusion also holds, we conclude that $\rho=\rho_N$, which finishes the proof.
\end{proof}

We now describe a model for $M(X^*, X\cup \overline{X^*})/\rho \simeq {\mathsf{FFR}}(X)$ based on the Cayley graph $\Cay(X^*, X\cup \overline{X^*})$. This model will be based on identifying a canonical representative $(\Gamma^{\wedge}, s)$ in the $\rho$-class of every $(\Gamma, s) \in M(X^*, X\cup \overline{X^*})$. The graph $\Gamma^{\wedge}$ will be obtained from $\Gamma$ by adding to it finitely many edges, using the transformations of types ($T_1$), ($T_2$), ($T_3$), which are presented below.\footnote{On figures, edges labeled by elements of $X$ are represented by solid lines, and  edges labeled by elements of $\overline{X^*}$ are represented by dashed lines.}

($T_1$) Suppose $\Gamma$ has
the edge $(v, \overline{u},vu)$, but does not have the edge $(v,\overline{\lambda},v)$. Define $\Gamma'$ by adding to $\Gamma$ the edge $(v,\overline{\lambda},v)$, see Figure~\ref{fig:c1}.

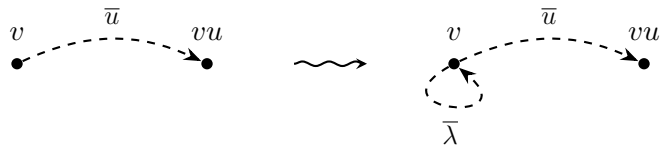
\begin{figure}[!ht]
\centering
\begin{tikzpicture}[scale=0.5]
\begin{scope}[every node/.style={circle,fill,inner sep=1.5pt}]
\node[label= {[label distance=0.1cm]:$v$}] (A) at (0,0) {};
\node[label= {[label distance=0.01cm]:{$v u$}}] (B) at (5,0) {}; 
\end{scope}

\begin{scope}[>={Stealth[black]}, every node/.style={fill=white,circle,scale=0.7}, every edge/.style={draw, thick}]
\path [->] (A) edge[font=\large, dashed, bend left=25] node[above=0.5pt] {$\overline{u}$} (B);
\end{scope}
 
\node (P) at (7, 0) {};
\node (Q) at (9.5, 0) {};
\begin{scope},
\draw[-stealth,thick,decorate,decoration={snake,amplitude=.3mm}] (P) -- (Q) node[above,midway]{};
\end{scope}

\begin{scope}[every node/.style={circle,fill,inner sep=1.5pt},xshift=11.5cm]
\node[label={[label distance=0.1cm]:$v$}] (A) at (0,0) {}; 
\node[label={[label distance=0.01cm]:{$v u$}}] (B) at (5,0) {}; 
\end{scope}

\begin{scope}[>={Stealth[black]}, every node/.style={fill=white,circle,scale=0.7}, every edge/.style={draw, thick}]
\path [->] (A) edge[font=\large, dashed, bend left=25] node[above=0.5pt] {$\overline{u}$} (B);
\path[->] (A) edge[font=\large, dashed, out=210,in=320,looseness=23]node[below=0.5pt] {$\overline{\lambda}$} (A);
\end{scope}
\end{tikzpicture}
\caption{Illustration of ($T_1$).}\label{fig:c1}
\end{figure}

($T_2$) Suppose $\Gamma$ has the edge $(v,x,vx)$,  but does not have the edge $(v, \overline{x}, vx)$. 
We define $\Gamma'$ by adding to $\Gamma$ the edge $(v, \overline{x},v x)$, see Figure~\ref{fig:c2a}.

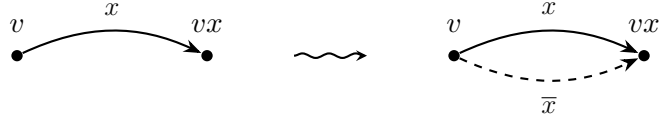
\begin{figure}[!ht]
\centering
\begin{tikzpicture}[scale=0.5]
\begin{scope}[every node/.style={circle,fill,inner sep=1.5pt}]
\node[label= {[label distance=0.1cm]:$v$}] (A) at (0,0) {};
\node[label= {[label distance=0.01cm]:{$v x$}}] (B) at (5,0) {}; 
\end{scope}

\begin{scope}[>={Stealth[black]}, every node/.style={fill=white,circle,scale=0.7}, every edge/.style={draw, thick}]
\path [->] (A) edge[font=\large, bend left=25] node[above=0.5pt] {$x$} (B);
\end{scope}
 
\node (P) at (7, 0) {};
\node (Q) at (9.5, 0) {};
\begin{scope},
\draw[-stealth,thick,decorate,decoration={snake,amplitude=.3mm}] (P) -- (Q) node[above,midway]{};
\end{scope}

\begin{scope}[every node/.style={circle,fill,inner sep=1.5pt},xshift=11.5cm]
\node[label={[label distance=0.1cm]:$v$}] (A) at (0,0) {}; 
\node[label={[
label distance=0.01cm]:{$v x$}}] (B) at (5,0) {}; 
\end{scope}

\begin{scope}[>={Stealth[black]}, every node/.style={fill=white,circle,scale=0.7}, every edge/.style={draw, thick}]
\path [->] (A) edge[font=\large, bend left=25] node[above=0.8pt] {$x$} (B);
\path [->] (A) edge[font=\large, dashed,bend right=25] node[below=0.8pt]{$\overline{x}$} (B);
\end{scope}
\end{tikzpicture}
\caption{Illustration of ($T_2$).}\label{fig:c2a}
\end{figure}

($T_3$) Suppose $\Gamma$ has the path labeled by $\overline{u}\,\overline{v}$ from $w$ to $w uv$, but does not have the edge $(w, \overline{uv},w uv)$. 
We define $\Gamma'$ by adding to $\Gamma$ the edge $(w, \overline{uv},w uv)$, see Figure~\ref{fig:c3a}.

\begin{figure}[!ht]
\begin{tikzpicture}[scale=0.9]
\begin{scope}[every node/.style={circle,fill,inner sep=1.5pt}]
\node[label= {[label distance=0.1cm]:$w$}] (A) at (0,0) {};
\node[label= {[label distance=0.01cm]:{$w u$}}] (B) at (2.5,0) {};
\node[label= {[label distance=-0.1cm]:{$w uv$}}] (C) at (5,0) {};
\end{scope}

\begin{scope}[>={Stealth[black]}, every node/.style={fill=white,circle,scale=0.7}, every edge/.style={draw, thick}]
\path [->] (A) edge[font=\large, dashed,bend left=25] node[above=0.5pt] {$\overline{u}$} (B);
\path [->] (B) edge[font=\large, dashed, bend left=25] node[above=0.5pt] {$\overline{v}$} (C);
\end{scope}

\node (P) at (6, 0) {};
\node (Q) at (8, 0) {};

\begin{scope},
\draw[-stealth,thick,decorate,decoration={snake,amplitude=.3mm}] (P) -- (Q) node[above,midway]{};
\end{scope}
\begin{scope}[every node/.style={circle,fill,inner sep=1.5pt},xshift=9cm]
\node[label={[label distance=0.1cm]:{$w$}}] (A) at (0,0) {}; 
\node[label={[label distance=0.01cm]:{$w u$}}] (B) at (2.5,0) {};
\node[label={[label distance=-0.1cm]:{$w uv$}}] (C) at (5,0) {};
\end{scope}

\begin{scope}[>={Stealth[black]}, every node/.style={fill=white,circle,scale=0.7}, every edge/.style={draw, thick}]
\path [->] (A) edge[font=\large, dashed,bend left=25] node[above=0.5pt] {$\overline{u}$} (B);
\path [->] (B) edge[font=\large, dashed, bend left=25] node[above=0.5pt] {$\overline{v}$} (C);
\path [->] (A) edge[font=\large, dashed, bend right=25] node[below=0.5pt] {$\overline{uv}$} (C);
\end{scope}

\end{tikzpicture}
\caption{Illustration of ($T_3$).}\label{fig:c3a}
\end{figure}
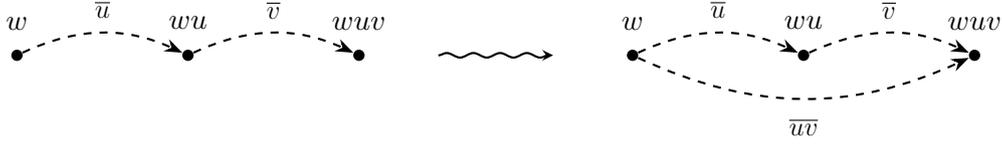

Let $\Gamma$ be a finite  accessible subgraph of $\Cay(X^*,X\cup \overline{X^*})$  with at least one edge. If any of the above described transformations can be applied to it, we (arbitrarily) choose such a transformation and apply it.
We then denote the obtained graph $\Gamma'$ by $\Gamma$ and repeat this step. Since there are at most two edges between any two vertices of $\Cay(X^* , X\cup \overline{X^*})$ and since $\Gamma$ is finite,  the process terminates in finitely many steps once none of the transformations of types ($T_1$), ($T_2$) or ($T_3$) is applicable to $\Gamma$. Moreover, the process is clearly locally confluent. Therefore, by Newman's lemma, it is confluent and terminates
at a finite accessible subgraph  $\Gamma^\wedge \supseteq \Gamma$  of $\Cay(X^* , X\cup \overline{X^*})$.

\begin{remark} \label{rem:closure}
It can be readily verified that the map $\Gamma \mapsto \Gamma^\wedge$ defines a closure operator on the semilattice $\mathcal{X}_{X\cup \overline{X^*}}$ of all finite accessible subgraphs of $\Cay(X^*, X\cup \overline{X^*})$ with at least one edge, that is:
\begin{enumerate}
\item $\Gamma^{\wedge}\supseteq \Gamma$;
\item $(\Gamma^{\wedge})^\wedge = \Gamma^\wedge$;
\item If $\Gamma_1\supseteq \Gamma_2$ then $\Gamma_1^{\wedge}\supseteq \Gamma_2^{\wedge}$.
\end{enumerate}
\end{remark}

Note that  $\Gamma^{\wedge}$ has a loop $(u,\overline{\lambda},u)$ at every its non-sink vertex $u$. Furthermore, the maximum projection of $M(X^*,X\cup\overline{X^*})/\rho$ is $(\Gamma_{\overline{\lambda}}, \lambda)$, where $\Gamma_{\overline{\lambda}}$ is the graph with the only vertex $\lambda$ and the only edge $(\lambda,\overline{\lambda},\lambda)$. 

As an example we describe $\Gamma_x^{\wedge}$ where $x\in X$ and $\Gamma_{\overline{u}}^{\wedge}$ where $u\in  X^+$. The graph $\Gamma_x^{\wedge}$ is obtained from the graph $\Gamma_x$ by adding to it the edges $(\lambda, \overline{x}, x)$ and $(\lambda, \overline{\lambda}, \lambda)$, see Figure \ref{fig:d4a}. The graph $\Gamma_{\overline{u}}^{\wedge}$ is obtained from the graph $\Gamma_{\overline{u}}$ by adding to it the edge $(\lambda, \overline{\lambda}, \lambda)$, see~Figure~\ref{fig:d4b}.

\begin{figure}[ht]
    \centering
    \begin{minipage}{0.45\textwidth}
        \centering
        \begin{tikzpicture}[scale=0.8]
\begin{scope}[every node/.style={circle,fill,inner sep=1.5pt}]
\node[label= {[label distance=-0.7cm]:$\lambda$}] (A) at (0,0) {};
\node[label= {[label distance=-0.7cm]:{$x$}}] (B) at (5,0) {}; 
\end{scope}

\begin{scope}[>={Stealth[black]}, every node/.style={fill=white,circle,scale=0.7}, every edge/.style={draw, thick}]
\path [->] (A) edge[font=\large, bend left=25] node[above=0.5pt] {$x$} (B);
\path [->] (A) edge[font=\large, dashed, bend right=25] node[above=0.5pt] {$\overline{x}$} (B);
\path[->] (A) edge[font=\large, dashed, out=160,in=70,looseness=30]node[above=1pt] {$\overline{\lambda}$} (A);
\end{scope}
\end{tikzpicture}
\caption{The graph $\Gamma_x^{\wedge}$.}\label{fig:d4a}
\end{minipage}
\hfill
\begin{minipage}{0.45\textwidth}
\centering
\begin{tikzpicture}[scale=0.8]
\begin{scope}[every node/.style={circle,fill,inner sep=1.5pt}]
\node[label= {[label distance=-0.7cm]:$\lambda$}] (A) at (0,0) {};
\node[label= {[label distance=-0.7cm]:{$u$}}] (B) at (5,0) {}; 
\end{scope}

\begin{scope}[>={Stealth[black]}, every node/.style={fill=white,circle,scale=0.7}, every edge/.style={draw, thick}]
\path [->] (A) edge[font=\large, dashed, bend left=25] node[above=0.5pt] {$\overline{u}$} (B);
\path[->] (A) edge[font=\large, dashed, out=160,in=70,looseness=30]node[above=1pt] {$\overline{\lambda}$} (A);
\end{scope}
\end{tikzpicture}
\caption{The graph $\Gamma_{\overline{u}}^{\wedge}$.}\label{fig:d4b}
\end{minipage}
\end{figure}

The following statement explains the rationale of the passage from $\Gamma$ to $\Gamma^{\wedge}$.

\begin{proposition}\label{prop:rho}
Let $(A,s), (B,t) \in M(X^*,X\cup\overline{X^*})$. Then
$(A,s) \mathrel{\rho} (B,t)$ if and only if $s=t$ and $A^\wedge = B^\wedge$.    
\end{proposition}

To prove Proposition \ref{prop:rho},
we need the following two lemmas.

\begin{lemma}\label{lem:n26b} For any elements $u,w,v$ of a restriction semigroup we have 
\begin{equation*}
(uw^+v)^+=(uw)^+(uv)^+.
\end{equation*}
\end{lemma}
\begin{proof}
The desired equality follows from the calculation:
\begin{align*}
(uw^+v)^+&= ((uw^+)^+uv)^+ &\text{(by \eqref{eq:aaa2})}\\
&=(uw^+)^+(uv)^+ &\text{(by \eqref{eq:id_rest})}\\
&= (uw)^+(uv)^+&\text{(by \eqref{eq:aaa1})}.
\end{align*}  
\end{proof}

\begin{lemma}\label{lem:n26e}
Suppose that $(A,s), (B,s) \in M(X^*,X\cup\overline{X^*})$ where $B\neq A$. For each $i\in \{1,2,3\}$ we have that 
$(B,s)$ is obtained from $(A,s)$ by an application of a rewriting rule from $\to_{N_i}$ if and only if $B$ is obtained from $A$ by an application of a transformation of type ($T_i$).
\end{lemma}

\begin{proof} 
Suppose that $(B,s)$ is obtained from $(A,s)$ by applying a  rewriting rule \mbox{$\overline{u}\to_{N_1}\overline{\lambda}\,\overline{u}$.} This means that in a decomposition of $(A,s)$ into generators (see \eqref{eq:aaa5}, \eqref{eq:n26c} and \eqref{eq:n26d}) there is a factor $(\Gamma_{\overline{u}},u)$, whereas replacing this factor by $(\Gamma_{\overline{\lambda}},\lambda)(\Gamma_{\overline{u}},u)$ gives a decomposition of $(B,s)$ into generators. But since $B\neq A$, this is equivalent to saying that $B$ is obtained from $A$ by an application of a transformation of type ($T_1$). For rewriting rules from $\to_{N_2}$ and $\to_{N_3}$ and transformations of types (T2) and (T3) the argument is similar.
\end{proof}

\begin{proof}[Proof of Proposition \ref{prop:rho}]
Let $(A,s), (B,t) \in M(X^*,X\cup\overline{X^*})$ be such that $(A,s) \mathrel{\rho} (B,t)$. 
Since $\rho \subseteq \sigma$, it follows that $s=t$. If $A=B$, then clearly $A^{\wedge} = B^{\wedge}$. So we can suppose that $A\neq B$. Given that the relation $\to_N$ generates $\rho$, there is $n\geq 1$ and a sequence
$$
(A,s)=(A_0,s), (A_1,s),\ldots, (A_n,s)=(B,s),
$$
such that for all $i=0,\dots, n-1$ we have 
$$(A_i,s) \to_N (A_{i+1},s) \quad {\text{ or }} \quad (A_{i+1},s) \to_N (A_{i},s).$$
We can suppose that $A_{i+1}\neq A_i$ for all $i=0,\dots, n-1$. It follows from Lemma \ref{lem:n26e} that for each such $i$ one of the graphs $A_i$ and $A_{i+1}$ is obtained from the other by applying one of the transformations described in ($T_1$), ($T_2$) and ($T_3$). Therefore, $A_i^{\wedge} = A_{i+1}^{\wedge}$ and thus $A^\wedge = B^\wedge$. 

Suppose now that $s=t$ and $A^{\wedge} = B^{\wedge}$. We  show that $(A,s) \mathrel{\rho} (A^{\wedge}, s)$. Since $A^{\wedge}$ is obtained from $A$ by a finite number of applications of  transformations of types ($T_1$), ($T_2$) or ($T_3$), Lemma \ref{lem:n26e} implies that $(A^{\wedge},s)$ is obtained from $(A,s)$ by a finite number of applications of rewriting rules from $\to_N$. Therefore, $(A,s) \mathrel{\rho} (A^{\wedge}, s)$, so that $(A,s) \mathrel{\rho} (A^{\wedge}, s) = (B^{\wedge}, s) \mathrel{\rho} (B, s)$.
\end{proof}

We say that a finite accessible subgraph $\Gamma$ of $\Cay(X^*, X\cup \overline{X^*})$ is {\em $\rho$-closed}, provided that $\Gamma = \Gamma^{\wedge}$.

Proposition \ref{prop:rho} shows that every $\rho$-class of $M(X^*, X\cup \overline{X^*})$ has a unique representative $(\Gamma, s)$ where $\Gamma$ is $\rho$-closed. We obtain the following result, which provides a model for ${\mathsf{FFR}}(X)\simeq M(X^*, X\cup \overline{X^*})/\rho$.

\begin{proposition}\label{prop:model}
Elements of ${\mathsf{FFR}}(X)$ can be identified with pairs 
$$(\Gamma, s) \in M(X^*, X\cup \overline{X^*})$$ where $\Gamma$ is a $\rho$-closed subgraph of $\Cay(X^*, X\cup \overline{X^*})$, with the operations given by 
\begin{equation*}
(\Gamma,s)(\Delta,t) = ((\Gamma \cup s\Delta)^\wedge, st), \qquad
(\Gamma,s)^+ = (\Gamma,\lambda), \qquad
\mx{(\Gamma,s)} = (\Gamma_{\overline{s}}^\wedge,s).
\end{equation*}
\end{proposition}

From now on, let ${\mathbb T}(X)$ denote the absolutely free algebra of the signature $(\cdot\,,^+,\mx{},\lambda)$. To formulate the result on the decidability of the word problem in $\mathsf{FFR}(X)$, we define the {\em complexity} ${\mathsf{c}}(u)$ of a term $u\in {\mathbb T}(X)$. 
By definition, the elements of $X$ and $\lambda$ (the constant) have complexity~$1$. If $u,v\in {\mathbb T}(X)$ are terms then ${\mathsf{c}}(u\cdot v) = {\mathsf{c}}(u) + {\mathsf{c}}(v) +1$ and ${\mathsf{c}}(\mx{u}) = {\mathsf{c}}(u^+) = {\mathsf{c}}(u)+1$. 

\begin{theorem}\label{thm:word}
The word problem for $\mathsf{FFR}(X)$ is decidable with an algorithm of the time complexity $O(n^3)$ where $n = {\mathrm{max}}\{{\mathsf{c}}(u), {\mathsf{c}}(v)\}$ and $u,v\in {\mathbb T}(X)$ are the input terms.
\end{theorem}

\begin{proof} 
Let $u\in {\mathbb T}(X)$ be a term. Suppose that $\mx{v}$, where $v\in {\mathbb T}(X)$, is a subterm of $u$. We view the value $[v]_{X^*}$ of $v$ in $X^*$ as an element of ${\mathbb T}(X)$. Hence $v$ and $[v]_{X^*}$ are terms in ${\mathbb T}(X)$ which have the same image under the canonical morphism to $X^*$. Therefore, these two elements are $\sigma$-related in $\mathsf{FFR}(X)$ and thus an application of the $\mx{}$-operation to both of these elements yields the same result. We conclude that the equality $\mx{v} = \mx{([v]_{X^*})}$ holds in $\mathsf{FFR}(X)$, so that we can simplify $u$ by replacing the subterm 
$\mx{v}$ by the subterm $\mx{([v]_{X^*})}.$ If, for example, $v=xy^+\mx{(\mx{(xz)}y^+)}x^+$ then  $[v]_{X^*} = xxz$. We can thus replace $\mx{(xy^+\mx{(\mx{(xz)}y^+)}x^+)}$ 
by $\mx{(xxz)}$. After several such steps,
we obtain a term $u' \in {\mathbb T}(X)$ such that any its subterm of the form $\mx{t}$ is such that $t\in X^*$.  We then create a new term, $\tilde{u}$, in the signature $(\cdot\,,^+)$ over the alphabet $X\cup \overline{X^*}$ by replacing every occurrence of 
$\mx{v}$, where $v\in X^*$, by $\overline{v}$ and replacing each $\lambda$ by $\overline{\lambda}$. Then we assign to $\tilde{u}$ its value $(\Gamma_{\tilde{u}}, [\tilde{u}]_{X^*})$ in $M(X^*, X\cup \overline{X^*})$. The value of $u$ in $M(X^*, X\cup \overline{X^*})/\rho$ is then the pair $(\Gamma_{\tilde{u}}^{\wedge}, [\tilde{u}]_{X^*})$.  

To check if terms $u$ and $v$ of ${\mathbb T}(X)$ have the same value in ${\mathsf{FFR}}(X)$, we check if  $[\tilde{u}]_{X^*} =  [\tilde{v}]_{X^*}$ and if  $\Gamma_{\tilde{u}}^{\wedge}=\Gamma_{\tilde{v}}^{\wedge}$. We construct $(\Gamma_{\tilde{u}}^{\wedge}, [\tilde{u}]_{X^*})$ from $u$ as is described in the previous paragraph:
\begin{enumerate}
\item Produce the term $u'$ in at most ${\mathsf{c}}(u)$ steps.
\item  Produce the term $\tilde{u}$ in at most ${\mathsf{c}}(u)$  steps. 
\item Produce the element $(\Gamma_{\tilde{u}}, [\tilde{u}]_{X^*})$ in  at most ${\mathsf{c}}(u)$ steps.
\end{enumerate}
The definition of $\Gamma_{\tilde{u}}$ implies that it has at most ${\mathsf{c}}(u)$ vertices. To get the graph $\Gamma_{\tilde{u}}^{\wedge}$ from $\Gamma_{\tilde{u}}$, for each vertex of the latter we possibly add a loop at this vertex, and for each pair of its distinct vertices we possibly add an edge from one of them to the other. This is done in at most ${\mathsf{c}}(u)$ + $\binom{{\mathsf{c}}(u)}{2}$ steps which is quadratic in ${\mathsf{c}}(u)$. It follows in particular that the graph $\Gamma_{\tilde{u}}^{\wedge}$ has $O(n^2)$ edges.

To check if $\Gamma_{\tilde{u}}^{\wedge} = \Gamma_{\tilde{v}}^{\wedge}$, it suffices to check the inclusion $\Gamma_{\tilde{u}}\subseteq \Gamma_{\tilde{v}}^{\wedge}$ and the dual inclusion. For each edge of $\Gamma_{\tilde{u}}$ we use at most $O(n^2)$ steps to check if it is an edge of $\Gamma_{\tilde{v}}^{\wedge}$. It follows that we check if the two graphs are equal in $O(n^3)$ steps. Furthermore, we check if $[\tilde{u}]_{X^*} =  [\tilde{v}]_{X^*}$ in $O(n)$ steps. The statement follows.
\end{proof}

\begin{remark}\label{rem:d1b}
To obtain the model of free $X$-generated $F$-restriction monoid, we must also take into account the relation which says that the generator $\overline{\lambda}$ is the identity element. Closed graphs then have loops labeled by $\overline{\lambda}$ at all of their vertices. To avoid these loops, we can exclude $\overline{\lambda}$ from the set of generators  and work with $M^{\lambda}(X^*, X \cup \overline{X^+})$.\footnote{This parallels the approach of our earlier work \cite{KLF25} about $F$-birestriction monoids.}
\end{remark}

\section{The free \texorpdfstring{$X$}{X}-generated strong \texorpdfstring{$F$}{F}-restriction semigroup}\label{sec:strong}

In this section, we adapt the approach of Section \ref{sec:F-restr} to obtain a model and solve the word problem for the free $X$-generated strong $F$-restriction semigroup $\mathsf{FFR}_s(X)$.

Similarly as in Lemma \ref{lem:free}, one shows that $\mathsf{FFR}_s(X)/\sigma \simeq X^*$. Remark \ref{rem:generated} implies that $\mathsf{FFR}_s(X)$ is  
$(X\cup \overline{X^*})$-generated as a restriction semigroup. Given that $\mathsf{FFR}_s(X)$ is $F$-restriction, it is proper by Proposition \ref{prop:n25}. The universal property of $M(X^*, X\cup \overline{X^*})$ now implies that
$\mathsf{FFR}_s(X)$ is an $(X\cup \overline{X^*})$-canonical $(\cdot\,,^+)$-quotient of $M(X^*, X\cup \overline{X^*})$. Let $\rho_s$ denote the $(\cdot\,,^+)$-congruence on $M(X^*, X\cup \overline{X^*})$ such that the quotient by it is isomorphic to $\mathsf{FFR}_s(X)$. Since $X^*$ is a canonical quotient of $\mathsf{FFR}_s(X)$, we have $\rho_s\subseteq \sigma$. 

By definition, $\rho_s$ is the coarsest $(\cdot\,, ^+)$-congruence on $M(X^*, X\cup \overline{X^*})$ which makes the quotient by it a strong $F$-restriction semigroup with the elements $\overline{u}\in \overline{X^*}$ being the maximum elements of their $\sigma$-classes. 
We consider the relation 
\begin{align*}
\to_{N_4} & =  \{(\overline{u}\,\overline{v}, \overline{u}^+\overline{uv})\colon u,v\in X^+\}
\end{align*}
on $M(X^*, X\cup \overline{X^*})$ and denote the union of the relations $\to_{N_1}$, $\to_{N_2}$ and $\to_{N_4}$ by $\to_{N_s}$.
Let $\rho_{N_s}$ be the congruence on $M(X^*,X\cup \overline{X^*})$ generated by $\to_{N_s}$. Since the latter congruence is contained in $\rho_s$, we have that $\rho_{N_s}\subseteq \rho_s$.

\begin{lemma}
Let $\rho$ be the congruence on $M(X^*,X\cup \overline{X^*})$, such that $M(X^*,X\cup \overline{X^*})/\rho\simeq \mathsf{FFR}(X)$, see Section \ref{sec:F-restr}. Then 
$\rho \subseteq \rho_{N_s}$.
\end{lemma}

\begin{proof}
Since $\rho$ is generated by $\to_N$, we need to show that $\to_N$ is contained in  $\rho_{N_s}$. Clearly, $\to_{N_1},\to_{N_2}$ are contained in $\to_{N_S}\subseteq\rho_{N_s}$, so it suffices to verify that $\to_{N_3}\subseteq \rho_{N_s}$. 
Suppose $\overline{u}\,\overline{v}\mathrel{\to_{N_3}}(\overline{u}\,\overline{v})^+\overline{uv}$. We calculate:
$$
\overline{u}\,\overline{v} = (\overline{u}\,\overline{v})^+\overline{u}\,\overline{v} \to_{N_4} (\overline{u}\,\overline{v})^+\overline{u}^+\overline{uv}= \overline{u}^+(\overline{u}\,\overline{v})^+\overline{uv}=(\overline{u}^+\overline{u}\,\overline{v})^+\overline{uv} = (\overline{u}\,\overline{v})^+\overline{uv},
$$
so that $\overline{u}\,\overline{v} \to_{N_4} (\overline{u}\,\overline{v})^+\overline{uv}$. It follows that
$(\overline{u}\,\overline{v},(\overline{u}\,\overline{v})^+\overline{uv})\in \rho_{N_s}$, as desired.
\end{proof}

\begin{proposition}\label{prop:mod2}
$M(X^*, X\cup \overline{X^*})/\rho_{N_s}$ is an $X$-generated strong $F$-restriction semigroup. Consequently, $\rho_s=\rho_{N_s}$ and thus $M(X^*, X\cup \overline{X^*})/\rho_{N_s}$ is $X$-canonically $(\cdot\,,^+,\mx{},\lambda)$-isomorphic to ${\mathsf{FFR}}_s(X)$.
\end{proposition}

\begin{proof}
Since $\rho \subseteq \rho_{N_s}$, $M(X^*, X\cup \overline{X^*})/\rho_{N_s}$ is an $X$-generated $F$-restriction semigroup and the elements $\overline{u}\in \overline{X^*}$ are the maximum elements of their $\sigma$-classes. We show that this quotient satisfies the identity $\mx{x}\mx{y} = (\mx{x})^+\mx{(xy)}$. We have $\mx{x} = \overline{\sigma^{\natural}(x)}$, $\mx{y} = \overline{\sigma^{\natural}(y)}$ and $\mx{(xy)} = \overline{\sigma^{\natural}(xy)} = \overline{\sigma^{\natural}(x)\sigma^{\natural}(y)}$ for all $x,y\in M(X^*, X\cup \overline{X^*})$. It follows that:
$$
\mx{x}\mx{y}  = \overline{\sigma^{\natural}(x)}\, \overline{\sigma^{\natural}(y)}
\mathrel{\rho_{N_s}} \overline{\sigma^{\natural}(x)}^+ \overline{\sigma^{\natural}(x)\sigma^{\natural}(y)} = (\mx{x})^+\mx{(xy)}.
$$
Therefore,  $M(X^*, X\cup \overline{X^*})/\rho_{N_s}$  is a  strong $F$-restriction semigroup. Hence, $\rho_s\subseteq \rho_{N_s}$, and the statement follows.
\end{proof}

We now describe a  model for $M(X^*, X\cup \overline{X^*})/\rho_s \simeq {\mathsf{FFR}}_s(X)$ based on the Cayley graph $\Cay(X^*, X\cup \overline{X^*})$.
Let $\Gamma$ be a finite accessible subgraph of $\Cay(X^*, X\cup \overline{X^*})$ with at least one edge. We introduce the following transformation of type ($T_4$):

($T_4$) Suppose that $\Gamma$ has the 
edges $(w, \overline{u}, w u)$ and $(w, \overline{uv}, w uv)$, but does not have the edge $(w u, \overline{v},w uv)$. We define
 $\Gamma'$ by adding to $\Gamma$ the edge $(w u, \overline{v},w uv)$, see Figure ~\ref{fig:c4a}.

\begin{figure}[!ht]
\begin{tikzpicture}[scale=0.9]
\begin{scope}[every node/.style={circle,fill,inner sep=1.5pt}]
\node[label= {[label distance=0.1cm]:$w$}] (A) at (0,0) {};
\node[label= {[label distance=0.01cm]:{$w u$}}] (B) at (2.5,0) {};
\node[label= {[label distance=-0.1cm]:{$w uv$}}] (C) at (5,0) {};
\end{scope}

\begin{scope}[>={Stealth[black]}, every node/.style={fill=white,circle,scale=0.7}, every edge/.style={draw, thick}]
\path [->] (A) edge[font=\large, dashed,bend left=25] node[above=0.5pt] {$\overline{u}$} (B);
\path [->] (A) edge[font=\large, dashed, bend right=25] node[below=0.5pt] {$\overline{uv}$} (C);
\end{scope}

\node (P) at (6, 0) {};
\node (Q) at (8, 0) {};

\begin{scope},
\draw[-stealth,thick,decorate,decoration={snake,amplitude=.3mm}] (P) -- (Q) node[above,midway]{
};
\end{scope}
\begin{scope}[every node/.style={circle,fill,inner sep=1.5pt},xshift=9cm]
\node[label={[label distance=0.1cm]:{$w$}}] (A) at (0,0) {}; 
\node[label={[label distance=0.01cm]:{$w u$}}] (B) at (2.5,0) {};
\node[label={[label distance=-0.1cm]:{$w uv$}}] (C) at (5,0) {};
\end{scope}

\begin{scope}[>={Stealth[black]}, every node/.style={fill=white,circle,scale=0.7}, every edge/.style={draw, thick}]
\path [->] (A) edge[font=\large, dashed,bend left=25] node[above=0.5pt] {$\overline{u}$} (B);
\path [->] (B) edge[font=\large, dashed, bend left=25] node[above=0.5pt] {$\overline{v}$} (C);
\path [->] (A) edge[font=\large, dashed, bend right=25] node[below=0.5pt] {$\overline{uv}$} (C);
\end{scope}

\end{tikzpicture}
\caption{Illustration of ($T_4$).}\label{fig:c4a}
\end{figure}
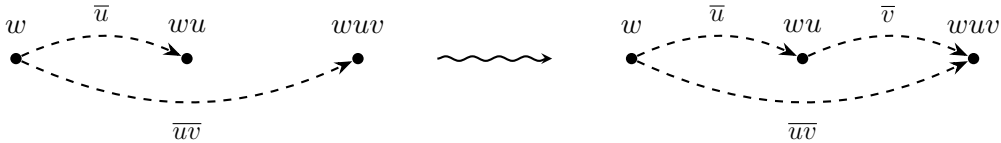

Just as in the previous section, if any of the transformations of type ($T_1$), ($T_2$), ($T_3$) or ($T_4$) can be applied to $\Gamma$, we (arbitrarily) choose such a transformation and apply it. We then denote the obtained graph $\Gamma'$ by $\Gamma$ and repeat this step. Similarly as before, the process is confluent and terminates in finitely many steps at a finite accessible subgraph  $\Gamma^{\wedge_s} \supseteq \Gamma$  of $\Cay(X^* , X\cup \overline{X^*})$.

The following proposition is analogous to Proposition~\ref{prop:rho}.

\begin{proposition}\label{prop:rhos}
Let $(A, a), (B, b) \in M(X^*,X\cup\overline{X^*})$. Then
$(A, a) \mathrel{\rho_s} (B, b)$ if and only if $ a=b$ and $A^{\wedge_s}=B^{\wedge_s}.$    
\end{proposition}

\begin{proof}
Let $(A, a), (B,b) \in M(X^*,X\cup\overline{X^*})$ be such that $(A, a) \mathrel{\rho_s} (B,b)$. As $\rho_{N_s}\subseteq \sigma$ it follows that $a=b$.
As in the proof of Proposition~\ref{prop:rho}, there exists $n \geq 1$ and a sequence
$$(A, a)=(A_0,a),\dots,(A_n, a)=(B, b),$$
such that for all $i=0,\dots, n-1$ we have
$$(A_i,a) \to_{N_s} (A_{i+1},a) \quad {\text{ or }} \quad (A_{i+1},a) \to_{N_s} (A_i,a),
$$ where $A_{i+1}\neq A_i$ for all $i$. 
It suffices to prove that $A_i^{\wedge_s} = A_{i+1}^{\wedge_s}$ for all $i\in\{0,\dots, n-1\}$.
Rewriting rules from $\to_{N_1}$ are treated in the same way as in the proof of Proposition~\ref{prop:rho}, and rewriting rules from $\to_{N_2}$ can be treated similarly. We are left to consider the rewriting rules from $\to_{N_4}$.
Suppose that $(A_{i+1}, a)$ is obtained from $(A_i, a)$ by applying a  rewriting rule $\overline{u}\,\overline{v} \to_{N_4} \overline{u}^+\overline{uv}$, that is, by replacing some factor $(\Gamma_{\overline{u}},u)(\Gamma_{\overline{v}},v)$ in a decomposition of $(A_i, a)$ into generators with $(\Gamma_{\overline{u}},\lambda)(\Gamma_{\overline{uv}},uv)$.
Applying ($T_3$) to the corresponding subgraph of $A_i$ and ($T_4$) to the corresponding subgraph of
$A_{i+1}$ results in the same subgraph, so that $A_i^{\wedge_s} = A_{i+1}^{\wedge_s}$. 
The case where $(A_{i},a)$ is obtained from $(A_{i+1},a)$ by applying a  rewriting rule $\overline{u}\,\overline{v} \to_{N_4} \overline{u}^+\overline{uv}$ is treated similarly.
Therefore, $A_i^{\wedge_s} = A_{i+1}^{\wedge_s}$, as desired.

Suppose now that $a=b$ and $A^{\wedge_s} = B^{\wedge_s}$. To show that $(A, a) \mathrel{\rho_s} (B,a)$, it suffices to show that $(A,a) \mathrel{\rho_s} (A^{\wedge_s},a)$. Since $A^{\wedge_s}$ is obtained from $A$ by finitely many applications of transformations of types ($T_1$), ($T_2$), ($T_3$) or ($T_4$).
Suppose $(A', a)$ for some graph $A'$ is obtained from $(A,a)$ by applying one of such transformations. We show that $(A, a) \mathrel{\rho_s} (A', a)$.
For transformations of types ($T_1$), ($T_2$) and ($T_3$), Lemma \ref{lem:n26e} implies that $(A,a) \mathrel{\rho} (A',a)$. 
Since $\rho \subseteq \rho_s$, we have $(A,a) \mathrel{\rho_s} (A', a)$. We finally suppose that $(A',a)$ is obtained from $(A,a)$ by applying a transformation of type ($T_4$). Then $(A, a)$ can be written via generators (see Proposition \ref {prop:M_generators}), so that the decomposition contains the subterm
$(\alpha\overline{u})^+(\alpha\overline{uv})^+$ = $(\alpha\overline{u})^+(\alpha\overline{uv})^+$ (where $\overline{u}$ stands for its value $(\Gamma_{\overline{u}},u)$ and $\alpha$ is some element of $M(X^*,X\cup\overline{X^*})$). Since this equals $((\alpha\overline{u})^+\alpha\overline{uv})^+ = ((\alpha\overline{u}^+)^+\alpha\overline{uv})^+ = (\alpha\overline{u}^+\overline{uv})^+$, we can assume that $(A, a)$ is written via generators, so that the decomposition contains the subterm $(\Gamma_{\overline{u}},\overline{u})^+(\Gamma_{\overline{uv}}, uv).$
Applying ($T_4$) to the corresponding subgraph of $A$, we get the graph $A'$ such that this subterm of $(A,a)$ gets replaced by the subterm $((\Gamma_{\overline{u}}, u)(\Gamma_{\overline{v}}, v))^+(\Gamma_{\overline{uv}}, uv)$ in $(A',a)$. Bearing in mind that
$$
\overline{u}\,\overline{v} \to_{N_4}\overline{u}^+\overline{uv}, \quad \overline{u}\,\overline{v} \to_{N_3}(\overline{u}\,\overline{v})^+\overline{uv},
$$
and that both $\to_{N_3}$ and $\to_{N_4}$ are contained in $\rho_s$, it follows that $(A, a)\mathrel{\rho_s} (A', a)$.
\end{proof}

Similarly as in Remark \ref{rem:closure}, the map $\Gamma \mapsto \Gamma^{\wedge_s}$ is a closure operator on the semilattice  ${\mathcal X}_{X\cup \overline{X^*}}$ 
of all finite accessible subgraphs of $\Cay(X^*,X\cup \overline{X^*})$ with at least one edge.

We say that a finite accessible subgraph $\Gamma$ of $\Cay(X^*, X\cup \overline{X^*})$ is {\em $\rho_s$-closed}, provided that  $\Gamma = \Gamma^{\wedge_s}$. 
By Proposition~\ref{prop:rhos}, each $\rho_{s}$-class in 
$M(X^*, X \cup \overline{X^*})$ contains a unique representative 
$(\Gamma, a)$, where $\Gamma$ is $\rho_s$-closed. This leads to the following proposition, which provides a model for ${\mathsf{FFR}}_s(X)\simeq M(X^*, X \cup \overline{X^*})/\rho_s$.

\begin{proposition} \label{prop:models}
Elements of ${\mathsf{FFR}}_s(X)$ can be identified with  pairs $(\Gamma, a)$, where $\Gamma$ is a $\rho_s$-closed subgraph of $\Cay(X^*, X\cup \overline{X^*})$, equipped with the operations
$$
(\Gamma, a)(\Delta,  b) = \bigl((\Gamma \cup s\Delta)^{\wedge_s},\, ab\bigr),
\qquad
(\Gamma,  a)^{+} = (\Gamma, \lambda), \qquad \mx{(\Gamma, a)}=(\Gamma_{ \overline{a}}^{\wedge_s}, a).
$$
\end{proposition}

In the next lemma we observe an important property of $\rho_s$-closed subgraphs of $\Cay(X^*, X\cup \overline{X^*})$.

\begin{lemma}\label{lem:d1a}
Let $\Gamma$ be a finite $\rho_s$-closed accessible subgraph of $\Cay(X^*,X\cup \overline{X^*})$ with at least one edge and let $u$, $uv$ be distinct vertices of $\Gamma$. Then $\Gamma$ contains the edge $(u, \overline{v}, uv)$.
\end{lemma}

\begin{proof} Since $\Gamma$ is accessible, there exists a directed path from $\lambda$ to $u$ and from $\lambda$ to $uv$. Since ($T_2$) and ($T_3$) can not be applied to $\Gamma$, it contains the edges $(\lambda,\overline{u},u)$ and $(\lambda,\overline{uv},uv)$. Further, since ($T_4$) can not be applied to $\Gamma$, it contains the edge $(u, \overline{v}, uv)$.
\end{proof}

\begin{theorem}\label{thm:words}
The word problem for $\mathsf{FFR}_s(X)$ is decidable with an algorithm of the time complexity $O(n^2)$ where $n = {\mathrm{max}}\{{\mathsf{c}}(u), {\mathsf{c}}(v)\}$ and $u,v\in {\mathbb T}(X)$ are the input terms.
\end{theorem}

\begin{proof}
Let $u\in{\mathbb T}(X)$ be a term over $X$ in the signature $(\cdot\,,^+,\mx{},\lambda)$. We consider the term $\tilde{u}$  in the signature $(\cdot\,,^+)$ over the alphabet $X\cup \overline{X^*}$ 
and its value $(\Gamma_{\tilde{u}},[\tilde{u}]_{X^*})$  in $M(X^*,X\cup \overline{X^*})$, as is described in the proof of Theorem~\ref{thm:word}. 
By Lemma \ref{lem:d1a}, 
if $\Gamma$ is a $\rho_s$-closed graph then any edge in $\Cay(X,X\cup \overline{X^*})$ labeled by an element of $\overline{X^+}$ is in $\Gamma$ whenever its incident vertices belong to $\Gamma$.
It follows that two $\rho_s$-closed subgraphs $\Gamma$ and $\Delta$ of $\Cay(X^*,X\cup \overline{X^*})$ coincide if and only if they have the same edges labeled by $X \cup \{\overline{\lambda}\}$.
If $\Gamma$ is a finite accessible subgraph of $\Cay(X^*,X\cup \overline{X^*})$ with at least one edge,  by $\Gamma'$ we denote the graph obtained from $\Gamma$ by removing all its edges labeled by elements of $\overline{X^+}$. We obtain a not necessarily connected graph (which may also have isolated vertices).
Then the terms $u$ and $v$ of ${\mathbb T}(X)$ represent the same element of $\mathsf{FFR}_s(X)$ if and only if 
$(\Gamma_{\tilde{u}})' = (\Gamma_{\tilde{v}})'$ and $[\tilde{u}]_{X^*} = [\tilde{v}]_{X^*}$.
As is argued in the proof of Theorem~\ref{thm:word}, we construct $\tilde{u}$ and $\tilde{v}$ in $O(n)$ steps. Since each of $\Gamma_{\tilde{u}}$ and $\Gamma_{\tilde{v}}$ has at most $O(n)$ edges, to obtain $(\Gamma_{\tilde{u}})'$ from $\Gamma_{\tilde{u}}$ and $(\Gamma_{\tilde{v}})'$ from $\Gamma_{\tilde{v}}$, one requires $O(n)$ steps. Given that the graphs $(\Gamma_{\tilde{u}})'$ and $(\Gamma_{\tilde{v}})'$ have each $O(n)$ edges and $O(n)$ vertices, we can check if they are equal in $O(n^2)$ steps.
\end{proof}

\begin{remark} 
To obtain the model of the free $X$-generated strong $F$-restriction monoid, we must also take into account the relation which says that the generator $\overline{\lambda}$ is the identity element. Similarly as in Remark \ref{rem:d1b}, we can exclude $\overline{\lambda}$ from the generators and obtain the model, such that the involved subgraphs $\Gamma$ do not have loops labeled by $\overline{\lambda}$ at all of their vertices. Going further and removing also all the edges labeled by elements of ${\overline{X^+}}$, we obtain the model for the free $X$-generated strong $F$-restriction monoid involving finite but not necessarily connected subgraphs of $\Cay(X^*,X)$. Since $\Cay(X^*,X)$ is a subgraph of $\Cay ({\mathsf{FG}}(X), X)$, where 
${\mathsf{FG}}(X)$ is the free $X$-generated group, it follows that the free $X$-generated strong $F$-restriction monoid $X$-canonically $(\cdot\,,^+,\mx{},\lambda)$-embeds into the free $X$-generated $F$-inverse monoid, described in \cite{AKSz21}.
\end{remark}

\section{The free \texorpdfstring{$X$}{X}-generated perfect \texorpdfstring{$F$}{F}-restriction monoid}\label{sec:perfect}
In this section we provide a  model of the free $X$-generated perfect $F$-restriction monoid\footnote{Recall that by Proposition \ref{prop:perfmonoid} a perfect 
$F$-restriction semigroups is necessarily a monoid.} $\mathsf{FFR}_p(X)$ and solve the word problem for it.

We start from the following useful observation.

\begin{lemma} \label{lem:p1} Let $R$ be an $X$-generated perfect $F$-restriction monoid. Then $R$ is $(X\cup \overline{X})$-generated as a restriction monoid.
\end{lemma}

\begin{proof}
As is pointed out in Remark \ref{rem:generated}, $R$ is $(X \cup \overline{S})$-generated as a restriction monoid, where $S=R/\sigma$. The maximum elements of the $\sigma$-classes of $R$ are the elements $\overline{u}$, where $u$ runs through $S$. Since $S$ is an $X$-generated monoid, we have that $u=\lambda$ or $u=x_1\cdots x_k$, where $k\geq 1$ and $x_i\in X$ for all $i\in \{1,\dots, k\}$. In the latter case we have  $\overline{u} = \overline{x_1}\cdots \overline{x_k}$ due to the identity $\mx{(xy)} = \mx{x}\,\mx{y}$. The statement follows.
\end{proof}

We are therefore seeking a model of $\mathsf{FFR}_p(X)$, which is a quotient of $M^{\lambda}(X^*,X\cup \overline{X})$.
Similarly as in Lemma \ref{lem:free}, we conclude that $\mathsf{FFR}_p(X)/\sigma \simeq X^*$. Since  $\mathsf{FFR}_p(X)$ is proper, the universal property of $M^{\lambda}(X^*, X \cup \overline{X})$ implies that $\mathsf{FFR}_p(X)$ is an $(X \cup \overline{X})$-canonical $(\cdot\,,^+)$-quotient of $M^{\lambda}(X^*,\, X \cup \overline{X})$.

Let $\rho_p$ be the $(\cdot\,,^+)$-congruence on $M^{\lambda}(X^*, X\cup \overline{X})$ such that $M^{\lambda}(X^*, X\cup \overline{X})/\rho_p$ is $(X\cup \overline{X})$-canonically isomorphic to $\mathsf{FFR}_p(X)$. That is, $\rho_p$ is the coarsest congruence on $M^{\lambda}(X^*, X\cup \overline{X})$ such that the quotient by it is an $X$-generated perfect $F$-restriction monoid and $\overline{x}$ is the maximum element of its $\sigma$-class for every $x\in X$. Similarly as in the previous sections, we have $\rho_p \subseteq \sigma$.
Since $M^{\lambda}(X^*, X \cup \overline{X})/\rho_p$ is a perfect $F$-restriction monoid, it must satisfy the following identities:
\begin{enumerate}
\item $\overline{x} \geq x$;
\item $\overline{x}^+ = \lambda$.
\end{enumerate}
We consider the relation
$$ 
\to_{N_p}=\{(x,x^+\overline{x}) \colon x\in X\}\cup \{(\lambda,\overline{x}^+) \colon x\in X\}
$$
on $M^{\lambda}(X^*, X \cup \overline{X})$ and denote by $\rho_{N_p}$ the congruence on $M^{\lambda}(X^*, X \cup \overline{X})$  generated by this relation. By definition, the relation $\to_{N_p}$ is contained in $\rho_p$, therefore  $\rho_{N_p}\subseteq \rho_p$.

\begin{proposition}
$M^{\lambda}(X^*, X\cup \overline{X})/\rho_{N_p}$ is an $X$-generated perfect $F$-restriction monoid. Consequently, $\rho_{N_p}=\rho_p$ and $M^{\lambda}(X^*, X\cup \overline{X})/\rho_{N_p}$ is $X$-canonically $(\cdot\,, ^+, \mx{}, \lambda)$-isomorphic to ${\mathsf{FFR}}_p(X)$.
\end{proposition}

\begin{proof}
Let $x_1\cdots x_n \in X^*$ be a non-empty word. We show that $\overline{x_1}\, \overline{x_2}\cdots \overline{x_n}$ is the maximum element of the $\sigma$-class of $M^{\lambda}(X^*, X\cup \overline{X})/\rho_{N_p}$ which projects onto this word. 
Let $u$ be an arbitrary element in this $\sigma$-class. By Lemma \ref{lem:generation}(1) we can write
$u=ev$, where $e$ is a projection and $v$ is a non-empty  term written in generators $X\cup \overline{X}$ using only the product operation.  Since $x\leq \overline{x}$ for all $x\in X$ and the natural partial order is compatible with the multiplication, it follows that
$u \leq v \leq \overline{x_1}\, \overline{x_2}\cdots \overline{x_n}$. 
Therefore, $\overline{x_1} \,\overline{x_2}\cdots \overline{x_n}$ is the maximum element of the $\sigma$-class which project onto $x_1\cdots x_n$. Moreover, as $\lambda$ is the maximum projection in $M^{\lambda}(X^*, X\cup \overline{X})$, it is also the maximum projection in  $M^{\lambda}(X^*, X\cup \overline{X})/\rho_{N_p}$.

Let further $u, v \in M^{\lambda}(X^*, X\cup \overline{X})/\rho_{N_p}$, where $\sigma^\natural (u)= x_1\cdots x_n$ and $\sigma^\natural (v)= y_1\cdots y_m$. The maximum element of the $\sigma$-class of $uv$ is then $\overline{x_{1}}\cdots \overline{x_{n}}\,\overline{y_{1}}\cdots \overline{y_{m}}$,
which is precisely the product of the maximum elements $\overline{x_{1}}\cdots \overline{x_{n}}$ and $\overline{y_{1}}\cdots \overline{y_{m}}$ of the $\sigma$-classes of $u$ and $v$.  
It follows that $M^{\lambda}(X^*, X \cup \overline{X})/\rho_{N_p}$ is perfect.  
Moreover, since $\overline{x} = \mx{x}$ for all $x\in X$, it is in fact $X$-generated as an $F$-restriction monoid.
Therefore, $\rho_{p} \subseteq \rho_{N_p}$.  
Given our previous observation that $\rho_{N_p}\subseteq \rho_{p}$, we conclude that $\rho_{p} = \rho_{N_p}$. This completes the proof.
\end{proof}

We now aim to provide a model of $M^{\lambda}(X^*, X \cup \overline{X})/\rho_p$ based on the Cayley graph $\Cay(X^*, X)$. 
We start from observing the following sufficient condition for two elements of
$M^{\lambda}(X^*, X \cup \overline{X})$ to be $\rho_p$-related.

\begin{lemma}\label{lem:n28b}
Let $(A,s) \in M^{\lambda}(X^*, X \cup \overline{X})$. Let further $v$ be a vertex of $A$ and  $B$ be a graph obtained from $A$ by adding to it the edge $(v,\overline{x},v x)$ (and the vertex $vx$, if required) for some $x\in X$.
Then $(A,s)\mathrel{\rho_p} (B,s).$
\end{lemma}

\begin{proof}
As in the proof of Lemma \ref{lem:n26e}, it can be shown that $(B,s)$ is obtained from $(A,s)$ by an application of a rewriting rule $\lambda\to_{N_p}\overline{x}^+$. Since $\to_{N_p}$ generates $\rho_p$, it follows that $(A,s)\mathrel{\rho_p} (B,s)$.
\end{proof}

We now characterize when $(A,s)$ and $(B,t)$ 
are $\rho_p$-related. 

\begin{proposition}\label{prop:n28c}
Let $(A,s), (B,t) \in M^{\lambda}(X^*, X \cup \overline{X})$.  
Then $(A,s) \mathrel{\rho_p} (B,t)$ if and only if $s = t$ and the sets of edges of $A$ and $B$ labeled by elements of $X$ coincide.
\end{proposition}

\begin{proof}
We first suppose that $(A,s) \mathrel{\rho_p} (B,t)$.
Since $\rho_p \subseteq \sigma$, we have that $s=t$. If $A=B$, the statement clearly holds. Suppose now that $A\neq B$. Since $\to_{N_p}$ generates $\rho_p$, there exists $n\geq 1$ and 
a sequence
$$
(A,s)=(A_0,s), (A_1,s),\ldots, (A_n,s)=(B,s),
$$
such that for all $i=0,\dots, n-1$ we have
$$(A_i,s) \to_{N_p} (A_{i+1},s) \quad {\text{ or }} \quad (A_{i+1},s) \to_{N_p} (A_i,s),
$$
where $A_{i+i}\neq A_i$ for all $i$.

Observe that an application of $x \to_{N_p} x^+\overline{x}$ or $\lambda \to_{N_p} \overline{x}^+$ results in adding an edge labeled by $\overline{x}$.
Since this does not affect edges labeled  by elements of $X$, we conclude that the edges of $A$ and $B$ labeled by elements of $X$ coincide.

We now suppose that the sets of edges of $A$ and $B$ labeled by elements of $X$ coincide.  
Since the graph $A \cup B$ has the same set of edges labeled by $X$ as $A$, it is obtained from $A$ by adding to it finitely many edges labeled by elements of $\overline{X}$.
Lemma \ref{lem:n28b}  implies that $(A,s) \mathrel{\rho_p} (A\cup B,s)$. By symmetry we also have $(B,s) \mathrel{\rho_p} (A\cup B,s)$. It follows that $(A,s) \mathrel{\rho_p} (B, s)$, as desired.
\end{proof}

We remark that if $v$ is a leaf of $A$ then $A$ does not have the edge $(v, \overline{x}, vx)$. By adding this edge to $A$ we obtain the graph $B$ which has one more vertex than $A$ such that $(A,s) \mathrel{\rho_p} (B,s)$ (as follows from Case 2 of the proof of Proposition \ref{prop:n28c}). This is in contrast with the situation of the previous two sections where the number of vertices of $A$ is the same for all the elements $(A,s)$ of a $\rho$-class (or a $\rho_s$-class). In particular, there is no analogue of $A^{\wedge}$ or $A^{\wedge_s}$ which would be a finite graph in the perfect case. 
 
To be able to provide a model for $M^{\lambda}(X^*, X \cup \overline{X})/\rho_p$ based on finite graphs, we make use of Proposition \ref{prop:n28c}. Let $(A,s)\in M^{\lambda}(X^*, X \cup \overline{X})$. We assign to $A$ its subgraph $\tilde{A}$ induced by the edges of $A$ labeled by $X$.
This is a finite but not necessarily connected subgraph of $\Cay(X^*,X)$ without isolated vertices (and which not necessarily contains the origin). Proposition \ref{prop:n28c} implies the following. 

\begin{corollary}\label{cor:m28}
Let $(A,s), (B,t) \in M^{\lambda}(X^*, X \cup \overline{X})$. Then $(A,s) \mathrel{\rho_p} (B,t)$ if and only if $s=t$ and $\tilde{A} = \tilde{B}$.
\end{corollary}

Let $\overline{\mathcal X}_X$ be the semilattice of all finite but not necessarily connected subgraphs of $\Cay(X^*, X)$ without isolated vertices (including the empty subgraph  $\varnothing$). Setting
\begin{equation}
{\mathsf{Perf}}(X)=\{(\Gamma,s)\colon \Gamma \in \overline{\mathcal X}_X\},
\end{equation}
the map $(\Gamma, s)\mapsto (\tilde{\Gamma},s)$ is a surjection from  $M^{\lambda}(X^*, X \cup \overline{X})$ to ${\mathsf{Perf}}(X)$. Due to Corollary \ref{cor:m28} it gives rise to a well-defined bijection $M^{\lambda}(X^*, X \cup \overline{X})/\rho_p  \to {\mathsf{Perf}}(X)$. We define the operations on ${\mathsf{Perf}}(X)$ so that they agree with those on $M^{\lambda}(X^*, X \cup \overline{X})/\rho_p$, as follows:
\begin{equation*}
(A,s)(B,t) = (A\cup sB,st), \quad (A,s)^+ = (A,\lambda), \quad \mx{(A,s)}=(\varnothing,s),
\end{equation*}
and the identity element is $(\varnothing,\lambda)$. 
We obtain the following result, which provides a model for ${\mathsf{FFR}}_p(X)$ based on the Cayley graph $\Cay(X^*,X)$.

\begin{theorem}\label{thm:modelp}
The map $M^{\lambda}(X^*,X\cup \overline{X})/\rho_p\to{\mathsf{Perf}}(X),\,\, [(\Gamma,s)]_{\rho_p}\mapsto (\tilde{\Gamma},s)$ is an $X$-canonical $(\cdot\,,^+,\mx{},\lambda)$-isomorphism. It follows that ${\mathsf{Perf}}(X)$ is $X$-canonically $(\cdot\,,^+,\mx{},\lambda)$-isomorphic to ${\mathsf{FFR}}_p(X)$.
\end{theorem} 

The definition of ${\mathsf{Perf}}(X)$ (or Theorem \ref{th:perf_structure}) yields that it decomposes as the semidirect product $\overline{\mathcal{X}}_X \rtimes X^*$, where the action of $X^*$ on $\overline{\mathcal{X}}_X$ is given by left translation: $u\cdot \Gamma = u\Gamma$ for all $u\in X^*$ and $\Gamma \in \overline{\mathcal{X}}_X$.

We now turn to the word problem for ${\mathsf{FFR}}_p(X)$.

\begin{theorem}\label{thm:wordp}
The word problem for $\mathsf{FFR}_p(X)$ is decidable with an algorithm of the time complexity $O(n^2)$, where $n = {\mathrm{max}}\{{\mathsf{c}}(u), {\mathsf{c}}(v)\}$ and $u,v\in {\mathbb T}(X)$ are the input terms.
\end{theorem}

\begin{proof}
Let $u\in {\mathbb T}(X)$ be a term over $X$ in the signature $(\cdot\,,^+,\mx{},\lambda)$. Suppose that $\mx{v}$, where $v\in {\mathbb T}(X)$, occurs as a subterm of $u$.  
If $[v]_{X^*} = \lambda$, then the value of $\mx{v}$ in ${\mathsf{FFR}}_p(X)$ is $\lambda$. If $[v]_{X^*}=x_1\cdots x_n$, where $n\geq 1$, then $\mx{v}$ and $\mx{x_1}\cdots \mx{x_n}$ have the same value in ${\mathsf{FFR}}_p(X)$. We replace $u$ by another term with the same value in ${\mathsf{FFR}}_p(X)$ by updating a subterm of the form $\mx{v}$ in the described way (if $u$ has such a subterm). In $O(n)$ steps we obtain a term  $u' \in {\mathbb T}(X)$ in which the operation $\mx{}$ is applied only to generators from $X$ and which has the same value in ${\mathsf{FFR}}_p(X)$ as $u$. We now create a new term, $\tilde{u}$, in the signature $(\cdot\,,^+,\lambda)$ over the alphabet $X\cup \overline{X}$ by replacing every occurrence of $\mx{x}$, where $x\in X$, by $\overline{x}$. We then assign to $\tilde{u}$ its value $(\Gamma_{\tilde{u}}, [\tilde{u}]_{X^*})$ in $M^ \lambda(X^*, X\cup \overline{X})$.

Similarly as in the proof of Theorem \ref{thm:word}, we need at most $O(n)$ steps to compare $[\tilde{u}]_{X^*}$ with $[\tilde{v}]_{X^*}$. We also need $O(n)$ steps to obtain the graph $
\Gamma_{\tilde{u}}$ and then again at most $O(n)$ steps to obtain $\tilde{\Gamma}_{\tilde{u}}$.
Since the latter graph has $O(n)$ edges, we can compare if $\tilde{\Gamma}_{\tilde{u}}$ and $\tilde{\Gamma}_{\tilde{v}}$ coincide in $O(n^2)$ steps.
\end{proof}

\begin{remark} 
Let $E$ be the semilattice of all finite but not necessarily connected subgraphs of the Cayley graph $\Cay({\mathsf{FG}}(X),X)$ without isolated vertices. Since ${\mathsf{Perf}}(X) \simeq \overline{\mathcal{X}}_X \rtimes X^*$, it   is contained in $E \rtimes {\mathsf{FG}}(X)$, which is the model of the free perfect $F$-inverse monoid ${\mathsf{FFI}}_p(X)$, see\cite{AKSz21}. Therefore,  ${\mathsf{FFR}}_p(X)$ $X$-canonically $(\cdot\,,^+,\mx{},\lambda)$-embeds into ${\mathsf{FFI}}_p(X)$.
\end{remark}

We conclude this section by a brief discussion of various approaches to the notion for the perfect $F$-birestriction monoid. 

\begin{remark} \label{rem:p} In our work \cite{KLF25} we defined an $F$-birestriction monoid $R$ to be perfect if it satisfies the identity $\mx{x}\mx{y} = \mx{(xy)}$ which says that the product of two $\sigma$-classes is a whole $\sigma$-class. Note that for $F$-inverse monoids (see \cite{AKSz21,KLF24}) this condition automatically implies that the inverse of a $\sigma$-class is a whole $\sigma$-class. Since $x^+ = xx^{-1}$ and $x^* = x^{-1}x$, it follows that the $^+$ and the $^*$ of a $\sigma$-class is a whole $\sigma$-class. It is readily seen that our model of the free perfect $F$-birestriction monoid given in \cite{KLF25} does not satisfy the conditions $(\mx{x})^+ = \mx{(x^+)}=\lambda$ and $(\mx{x})^* = \mx{(x^*)}=\lambda$. It seems, however, natural to call an $F$-birestriction monoid {\em perfect} if these two conditions are imposed. The version of the perfect $F$-birestriction monoids of \cite{KLF25} is then naturally termed {\em weakly perfect}.  Denoting by ${\mathsf{FFBR}}_{p}(X)$  the free $X$-generated perfect $F$-birestriction monoid (with the identities $(\mx{x})^+ = \mx{(x^+)}=\lambda$ and $(\mx{x})^* = \mx{(x^*)}=\lambda$ included in the definition), one can show that ${\mathsf{FFBR}}_{p}(X)$ is isomorphic to the semidirect product $E \rtimes X^*$ (where $E$ is the same as in the previous remark).
\end{remark}

We obtain a series of inclusions $
\overline{\mathcal{X}}_X \rtimes X^* \subseteq E\rtimes X^* \subseteq E \rtimes {\mathsf{FG}}(X)$,
which illustrates the $X$-canonical embeddings
$$
{\mathsf{FFR}}_p(X) \hookrightarrow {\mathsf{FFBR}}_{p}(X) \hookrightarrow {\mathsf{FFI}}_p(X).
$$

\section{The free \texorpdfstring{$X$}{X}-generated \texorpdfstring{$F$}{F}-restriction semigroup and its strong and perfect analogues are left ample}\label{sec:ample}

\begin{definition}
A restriction semigroup will be called {\em left ample}\footnote{We retain the term left ample for consistency with the literature, see \cite{GG00,G96}.}
if it satisfies the quasi-identity
\begin{equation}\label{eq:ample}
xz = yz \;\rightarrow\; xz^{+} = yz^{+}.    
\end{equation}    
\end{definition}

\begin{proposition}\label{prop:d4a}
${\mathsf{FFR}}(X)$, ${\mathsf{FFR}}_s(X)$ and ${\mathsf{FFR}}_p(X)$ are left ample.   
\end{proposition}

\begin{proof} We prove the statement for 
${\mathsf{FFR}}(X)$, the arguments for ${\mathsf{FFR}}_s(X)$ and ${\mathsf{FFR}}_p(X)$ being similar.
We work with the canonically isomorphic copy $M(X^*, X\cup \overline{X^*})/\rho$ of ${\mathsf{FFR}}(X)$. 
Its elements are canonical representatives of $\rho$-classes of the form $(A,a)$ where $A$ is a $\rho$-closed subgraph of $\Cay(X^*,X\cup \overline{X^*})$. 
We need to show that $M(X^*, X\cup \overline{X^*})/\rho$ satisfies the quasi-identity \eqref{eq:ample}.
Let $(A,a),(B,b),(C,c)\in M(X^*, X\cup \overline{X^*})/\rho$ and suppose $(A,a)(C,c)=(B,b)(C,c)$. By the definition of multiplication in Proposition \ref{prop:model}, this implies that $(A\cup aC)^{\wedge}=(B\cup bC)^{\wedge}$ and $ac=bc$. Then $a=b$ and 
$$
(A,a)(C,c)^+= ((A\cup aC)^{\wedge},a)=((B\cup bC)^{\wedge},b) = (B,b)(C,c)^+,
$$
as required.
\end{proof}

Let ${\mathbf{FR}}^a$ be the quasi-variety in the signature $(\cdot\,,^+,\mx{},\lambda)$ defined by all the identities which define the variety ${\mathbf{FR}}$ of $F$-restriction semigroups with the maximum projection $\lambda$ and the additional quasi-identity \eqref{eq:ample}. Similarly, let ${\mathbf{FR}}_s^a$ be the quasi-variety ${\mathbf{FR}}^a \cap {\mathbf{FR}}_s$ and ${\mathbf{FR}}_p^a$ be the quasi-variety ${\mathbf{FR}}^a \cap {\mathbf{FR}}_p$.

\begin{proposition}\label{prop:ample}
The following statements hold:
\begin{enumerate}
\item The free $X$-generated object of the quasi-variety ${\mathbf{FR}}^a$ coincides with ${\mathsf{FFR}}(X)$.
\item The free $X$-generated object of the quasi-variety ${\mathbf{FR}}_s^a$ coincides with ${\mathsf{FFR}}_s(X)$. 
\item The free $X$-generated object of the quasi-variety ${\mathbf{FR}}_p^a$ coincides with ${\mathsf{FFR}}_p(X)$.
\end{enumerate}
\end{proposition}

\begin{proof}
(1) Let $F_X$ be the free $X$-generated object of the quasi-variety ${\mathbf{FR}}^a$. By definition, we have that $F_X$ is a canonical quotient of ${\mathsf{FFR}}(X)$. By Proposition \ref{prop:d4a} we have that
${\mathsf{FFR}}(X)$ is left ample, so that ${\mathsf{FFR}}(X)$ is a canonical quotient of~$F_X$. The statement follows.

Parts (2) and (3) follow similarly.
\end{proof}

\section*{Acknowledgments}
We are grateful to the referee for the careful reading of the paper and for the constructive suggestions that helped us to improve the exposition.

\footnotesize
\def\bibspacing{-3pt}

\end{document}